\documentclass[12pt]{amsart}

\newcommand{\Rold}{R}
\newcommand{\CPICo}{{\ensuremath{\mathrm{C_{PIC1}}}}}
\newcommand{\CPICt}{\ensuremath{\mathrm{C_{PIC2}}}}
\newcommand{\cone}{\ensuremath{\mathrm{C}}}

\usepackage{mathrsfs} 

\usepackage{wasysym,stmaryrd}

\usepackage{cancel}
\usepackage{latexsym, amssymb, amsmath}
\usepackage{soul,esint}
\usepackage{mathtools}
\mathtoolsset{showonlyrefs}
\usepackage{amsfonts, graphicx}
\usepackage{graphicx,color}
\newcommand{\be}{\begin{equation}}
\newcommand{\ee}{\end{equation}}
\newcommand{\beq}{\begin{eqnarray}}
\newcommand{\eeq}{\end{eqnarray}}

\newtheorem{claim}{Claim}[section]

\newtheorem{thm}{Theorem}[section]

\newtheorem{lma}[thm]{Lemma}
\newtheorem{prop}[thm]{Proposition}
\newtheorem{cor}[thm]{Corollary}
\newtheorem{defn}[thm]{Definition}

\theoremstyle{remark}
\newtheorem{rem}[thm]{Remark}
\numberwithin{equation}{section}

\def\be{\begin{equation}}
\def\ee{\end{equation}}
\def\bee{\begin{equation*}}
\def\eee{\end{equation*}}

\def\lf{\left}
\def\ri{\right}

\def\Ric{\text{\rm Ric}}

\def\p{\partial}

\def\rheat{\lf(\frac{\p}{\p t}-\Delta_{g(t)}\ri)}

\def\tr{\operatorname{tr}}

\def\e{\varepsilon}
\def\a{{\alpha}}
\def\b{{\beta}}

\def\Rold{{R}}

\def\C{\mathbb{C}}
\newcommand{\R}{\ensuremath{{\mathbb R}}}

\begin{document}
\title[]
{Manifolds with PIC1 pinched curvature}

\author{Man-Chun Lee}
\address[Man-Chun Lee]{Department of Mathematics, The Chinese University of Hong Kong, Shatin, N.T., Hong Kong
}
\email{mclee@math.cuhk.edu.hk}

\author{Peter M. Topping}
\address[Peter M. Topping]{Mathematics Institute, Zeeman Building, University of Warwick, Coventry CV4 7AL}
 \email{P.M.Topping@warwick.ac.uk}

\renewcommand{\subjclassname}{
  \textup{2020} Mathematics Subject Classification}

\newcommand{\idorg}{g}

\date{12 May 2025}

\begin{abstract}
Recently it has been proved \cite{LeeTopping2022,DSS22,Lott2019} that three-dimensional complete manifolds with non-negatively pinched Ricci curvature must be flat or compact, thus confirming a conjecture of Hamilton. In this paper we generalise our work on the existence of Ricci flows from non-compact pinched three-manifolds in order to prove a higher-dimensional analogue. 
We construct a solution to Ricci flow, for all time, starting with an
arbitrary complete non-compact manifold that is PIC1 pinched. As an application
we prove that any complete manifold of non-negative complex sectional curvature that is PIC1 pinched must be flat or compact.
\end{abstract}


\maketitle{}

\markboth{}{Manifolds with PIC1 pinched curvature} 
\section{introduction}
\label{intro_sect}

\newcommand{\scal}{\mathrm{scal}}

In \cite{LeeTopping2022}, the following result was proved by lifting a remaining additional hypothesis of bounded curvature from the work of Deruelle-Schulze-Simon 
\cite{DSS22}, which in turn appealed to work of Lott  \cite{Lott2019}.

\begin{thm}[{Hamilton's pinching conjecture, cf. \cite[Conjecture 3.39]{CLN09}}]
\label{3D_thm}
Suppose $(M^3,g_0)$ is a complete (connected) three-dimensional Riemannian manifold with $\Ric\geq \e\, \scal\geq 0$ for some $\e>0$. Then  $(M^3,g_0)$ is either flat or  compact.
\end{thm}

In this paper we develop a higher-dimensional version of the Ricci flow existence theory we established in \cite{LeeTopping2022}, in order to prove a pinching result in general dimension.
In order to state the result, we need to understand a little about the notion of isotropic curvature.
See also \cite{PIC1_survey} for a different perspective.

Denote the space of algebraic curvature tensors on $\R^n$ by $\mathcal{C}_B(\mathbb{R}^n)$.
Given $R\in \mathcal{C}_B(\mathbb{R}^n)$, we can extend it by complex linearity to $\C^n$. 
Although we always have the symmetries of $R$ in mind, we view it as a $(0,4)$ tensor to avoid ambiguities of normalisation.
To each two-complex-dimensional subspace $\Sigma$ of $\C^n$ we can then associate a complex sectional curvature. Concretely, if $v,w\in \C^n$ give an orthonormal basis of $\Sigma$, then the complex sectional curvature associated with $\Sigma$ is
$R(v,w,\bar{v},\bar{w})$.  We say that $R\in \mathcal{C}_B(\mathbb{R}^n)$
has non-negative complex sectional curvature if all these curvatures are non-negative, whichever $\Sigma$ we choose, and we denote the cone of all such curvature tensors
by \CPICt{}.
Because this condition is $O(n)$-invariant, we can talk of a manifold having 
non-negative complex sectional curvature if this property holds for the curvature tensor at every tangent space.

We can weaken the notion of curvature positivity by restricting the sections $\Sigma$ that we consider. In particular, one can ask for non-negativity of complex sectional curvature only for  PIC1 sections, defined to be those $\Sigma$ that contain some nonzero vector $v$ 
whose conjugate $\bar v$ is orthogonal to the entire section 
$\Sigma$. 
The algebraic curvature tensors $R$ with non-negative complex sectional curvature for each such restricted $\Sigma$ form a cone we denote by \CPICo. 

Asking that a curvature tensor lies in \CPICo{} or \CPICt{} is asking that certain natural curvature averages are non-negative.

For $n\geq 4$, one can describe the cones \CPICo{} and \CPICt{} more explicitly 
as follows 
(cf. \cite[Propositions 7.14 and 7.18]{BrendleBook}).
The cone \CPICt{} is the cone consisting of curvature tensors 
$R$ 
satisfying 
$$R_{1313}+\lambda^2 R_{1414}+\mu^2 R_{2323}+\lambda^2\mu^2 R_{2424}-2\lambda\mu R_{1234}\geq 0$$
for all orthonormal four-frames $\{e_1,e_2,e_3,e_4\}\subset \mathbb{R}^n$ and all $\lambda,\mu\in [0,1]$. Similarly, \CPICo{} is the cone of curvature tensors
satisfying 
$$R_{1313}+\lambda^2 R_{1414}+ R_{2323}+\lambda^2 R_{2424}-2\lambda R_{1234}\geq 0$$
for all orthonormal four-frames $\{e_1,e_2,e_3,e_4\}\subset \mathbb{R}^n$ and all $\lambda\in [0,1]$.

\begin{thm}[Main theorem]
\label{main-thm-updated}
Suppose $(M^n,g_0)$ is a complete manifold of non-negative complex sectional curvature with $n\geq 3$  that is pinched in the sense that
\begin{equation}
\label{PIC1pinching}
R_{g_0}-\e_0\,\mathrm{scal}(R_{g_0})\cdot I\in \CPICo
\end{equation}
for some 
$\e_0>0$.
Then $(M,g_0)$ is either flat or compact.
\end{thm}

To clarify, we denote by $R_{g}$ the curvature tensor of a Riemannian metric $g$, 
with \eqref{PIC1pinching}
holding in every tangent space.
Given 
$R\in \mathcal{C}_B(\mathbb{R}^n)$, we  write 
$\Ric(R)_{ik}:=R_{ijkj}$ and $\scal(R):=\Ric(R)_{ii}$ in order to match the usual notions of Ricci and scalar curvature exactly.
Meanwhile, $I$ is the curvature tensor defined by
\begin{equation}
\label{I_def}
I_{ijkl}=\delta_{ik}\delta_{jl}-\delta_{il}\delta_{jk}.
\end{equation}
Thus on the unit round sphere $S^n$ we would have $R_g=I$.

Explicitly, our pinching condition \eqref{PIC1pinching} is saying 
exactly that when we compute the complex sectional curvature corresponding to a PIC1 section then not only should it be non-negative, it should also be bounded below by $\e_0\,\scal(R_{g_0})$.

In arbitrary dimension, every $R\in\CPICo$ automatically has non-negative Ricci curvature, 
and for $n=3$ the conditions are equivalent (see, for example, \cite[Section 2]{MT2020}).
For $n=3$, the PIC1 pinching condition \eqref{PIC1pinching} considered in the theorem is equivalent to pinched Ricci curvature. 
This case is already handled by the theory we developed in 
\cite{LeeTopping2022} in that dimension, and thus we may focus here on the case that $n\geq 4$.

Our main theorem extends several earlier pinching results.
Brendle and Schoen \cite[Theorem 7.4]{BSsurvey} proved that if $(M^n,g_0)$ is assumed,
in addition to non-negative complex sectional curvature, to satisfy a stronger PIC2 pinching condition, that is, one assumes 
$$R_{g_0}-\e_0\,\mathrm{scal}(R_{g_0})\cdot I\in \CPICt,$$ 
for  some $\e_0>0$, and additionally one assumes that the sectional curvature is uniformly bounded throughout, and that the scalar curvature is strictly positive,
then $(M,g_0)$ is compact.
In turn, this generalises earlier work of Chen and Zhu \cite{ChenZhu2000}
and of Ni and Wu \cite{NiWu2007}.

The strategy we use to prove Theorem \ref{main-thm-updated} is to assume, contrary to the theorem, that our manifold is neither flat nor compact, 
and then to 
flow $(M,g_0)$ under Ricci flow for all time and to analyse its asymptotic behaviour. 
The rough idea that echoes what has been done in previous work 
starting with Chen-Zhu \cite{ChenZhu2000}
is that 
parabolic blow-downs of the Ricci flow would like to settle down to an expanding gradient soliton, but that this is incompatible with the 
(scale-invariant)
pinching condition. 
The main challenge that we address is to establish the following existence theorem for Ricci flow, starting with a manifold of possibly unbounded curvature, that mirrors our earlier theory from \cite{LeeTopping2022}.

\begin{thm}\label{Thm:existence}
For any $n\geq 4$ and $\e_0\in (0,\frac{1}{n(n-1)})$, there exist $a_0>0$ and
$\e_0'\in (0,\frac{1}{n(n-1)})$ such that 
the following holds. Suppose $(M^n,g_0)$ is a complete non-compact manifold 
such that
\begin{equation}
\label{e0_pinching_hyp}
R_{g_0}-\e_0\, \mathrm{scal}(R_{g_0}) \cdot I\in \CPICo
\end{equation}
on $M$. Then there exists a smooth complete Ricci flow solution $g(t)$ on $M\times [0,+\infty)$ such that for all $t>0$,
\begin{enumerate}
\item[(a)] $R_{g(t)}-\e_0'\, \mathrm{scal}(R_{g(t)}) \cdot I\in \CPICo$; 
\item[(b)] $|R_{g(t)}|\leq a_0 t^{-1}$.
\end{enumerate}
\end{thm}
Note that the restriction $\e_0<\frac{1}{n(n-1)}$, together with the pinching hypothesis, implies that 
$\mathrm{scal}(R_{g_0})\geq 0$.

This existence theorem only requires PIC1 pinching, with no requirement for non-negative complex sectional curvature. A theory for Ricci flow starting with general open manifolds with non-negative complex sectional curvature was developed by 
Cabezas-Rivas and Wilking \cite{CabezasWilking2015}. Although we do have that hypothesis at our disposal 
in the application, we must develop a new theory in order to obtain both long-time existence and the required curvature decay, both of which require
the PIC1 pinching hypothesis. We emphasise that for the existence theory
we do not make
any non-collapsing assumption, we do not have any boundedness of curvature assumption, and we do not assume 
non-negativity of the complex sectional curvature.

\begin{rem}
\label{PICt_PICo_rmk}
Our PIC1 pinching theorem raises the question of whether the cone in which 
the curvature of $g_0$ is assumed to lie can be weakened from $\CPICt$ to $\CPICo$
in Theorem \ref{main-thm-updated}.
Our Ricci flow existence theory already works in this more general situation, but having some negative complex sectional curvature complicates the required blow-down argument in the proof of Theorem \ref{pinching_was_cor}.
\end{rem}

If we combine our main result with earlier work of B\"ohm-Wilking \cite{BohmWilking2008} and 
Brendle \cite{BrendleBook}, we rapidly obtain the following consequence.
\begin{cor}
\label{main-thm-cor}
Suppose $(M^n,g_0)$ is a complete manifold of non-negative complex sectional curvature with $n\geq 3$  that is not everywhere flat, and is pinched in the sense of 
\eqref{PIC1pinching} for some $\e_0>0$.
Then $(M,g_0)$ is diffeomorphic to a spherical space form.
\end{cor}
We will give the proof of Corollary \ref{main-thm-cor} towards the end of Section \ref{exist_sect}.

Our main existence theorem \ref{Thm:existence} fits within several active research areas. 
The techniques we develop build on  ideas introduced to construct complete Ricci flows starting with smooth manifolds with possibly unbounded curvature in other settings, necessarily on non-compact domains.  This is achieved by the construction of local Ricci flows on an exhaustion of the domain, or alternative approximations, satisfying appropriate estimates that allow one to pass to a limit to give a global solution. 

The first instance of this was the general theory of Ricci flow in two dimensions 
\cite{JEMS, GT2}.
In higher dimensions one can run the flow starting with manifolds of nonnegative complex sectional curvature \cite{CabezasWilking2015}. A common theme, following initial work of  Simon \cite{MilesCrelle3D},  is to flow starting  with a manifold that is globally non-collapsed and satisfies a lower bound on an appropriate notion of curvature.
Examples of such curvature notions include Ricci curvature in three dimensions
\cite{Hochard, ST2}, complex sectional curvature \cite{BamlerCabezasWilking2019}, the curvature associated with PIC1 \cite{yi_lai} and, in the setting of K\"ahler manifolds, a combination of Ricci and orthogonal bisectional curvatures \cite{LeeTamOB}. 
Alternative interesting hypotheses include a 
Ricci lower bound and local almost-Euclidean isoperimetric inequality 
\cite{HeFei}, 
and appropriate curvature lower bounds together with smallness of an $L^1$ Morrey-type norm of the curvature (i.e. smallness of rescaled averages of the curvature over balls of varying radius) \cite{ChanLeeGap}.
See also \cite{Chan_Lee_Huang} and \cite{Liang}.
The case of $U(n)$-invariant initial metrics on $\C^n$ was handled in 
\cite{CLT17}.

Note that many of the existence theorems above rely in an essential way on a non-collapsing assumption. Such an assumption  is not permissible in our application
and the way in which we deal with this is one of the novelties of our work.

A related endeavour to the existence theory described above is the programme to use the Ricci flow to smooth out rough initial data.
An early instance of this was Simon's work to smooth  $C^0$ metrics \cite{simonC0};
cf. later work in the K\"ahler setting \cite{ChauLee20}.
In this paper our initial data is smooth, but we can draw on techniques from some of the papers that start the flow with some form of limit space 
(following \cite{MilesCrelle3D})
in the non-compact setting;
see, for example, \cite{Hochard, ST2, yi_lai, Lee_Tam_JDG, MT2020}.
Although less related to this paper, a number of other Ricci flow papers handle rough initial data of various other types, including \cite{KochLamm, LammSimon, ChuLee,
TY3}.

\vskip 0.2cm
%
\emph{Acknowledgements:} 
MCL was partially supported by Hong Kong RGC grant (Early Career Scheme) of Hong Kong No. 24304222 and a direct grant of CUHK.
PMT was supported by EPSRC grant EP/T019824/1.
For the purpose of open access, the authors have applied a Creative Commons Attribution (CC BY) licence to any author accepted manuscript version arising.

\section{Pinching cones}
\label{pinching_cones_sect}

In this section, we briefly survey some general tools used throughout this paper that are largely taken from 
\cite{BrendleHuiskenSinestrari2011,BohmWilking2008}. 
For simplicity we work in dimension $n\geq 4$.

Given 
$R\in \mathcal{C}_B(\mathbb{R}^n)$,  define a new algebraic curvature tensor $Q(R)$ by
\begin{equation}
\label{Qdef}
Q(R)_{ijkl}=R_{ijpq}R_{klpq}+2R_{ipkq}R_{jplq}-2R_{iplq}R_{jpkq}
\end{equation}
and consider the `Hamilton ODE': $\frac{d}{dt}R =Q(R)$. For the significance of this ODE, and the geometry behind the quantity $Q(R)$, one can refer to \cite{BohmWilking2008}.

We are interested in cones 
$\cone\subset \mathcal{C}_B(\mathbb{R}^n)$ 
that are
invariant under the Hamilton ODE, such as
\CPICo{} and \CPICt{} \cite{BrendleSchoen2009, wilking2013}.
We will consider  cones $\cone$ with the following properties: 

\begin{enumerate}
\item[\bf(I)] $\cone$ is closed, convex, and $O(n)$-invariant;
\item[\bf(II)] $\cone$ is transversally invariant under the Hamilton ODE;
\item[\bf(III)] Every 
$R\in \cone\setminus \{0\}$ has positive scalar curvature;
\item [\bf(IV)]The curvature tensor $I$ defined in \eqref{I_def}
lies in the interior of $\cone$.
\end{enumerate}
In {\bf (II)}, transversally invariant is equivalent to saying that $Q(R)$ lies in the interior of the tangent cone $T_R\cone$ for all $R\in \cone\setminus \{0\}$. 

\medskip

In the following, we will recall two 
families of cones that satisfy {\bf (I)-(IV)}, adopting notation from \cite{BrendleBook} and \cite{BohmWilking2008} where appropriate.
We start by recalling the endomorphism $\ell_{a,b}$ on 
$\mathcal{C}_B(\mathbb{R}^n)$ 
introduced by B\"ohm and Wilking \cite{BohmWilking2008}. 
For any $a,b\in \mathbb{R}$ and $R\in  \mathcal{C}_B(\mathbb{R}^n)$, we define
\begin{equation*}
\begin{aligned}
\ell_{a,b}(\Rold) &=\Rold+b\cdot  \left(\Ric(R)-\frac{\mathrm{scal}(\Rold)}{n}g\right)\owedge g+\frac{a}n \mathrm{scal}(\Rold) \cdot  g\owedge g\\
&=\Rold+b\cdot  \Ric(R)\owedge g+\frac{a-b}n \mathrm{scal}(\Rold) \cdot  g\owedge g
\end{aligned}
\end{equation*}
where $g$ is the metric ($g_{ij}=\delta_{ij}$ on $\R^n$) and $\owedge$ denotes the Kulkarni-Nomizu product:
$$(A \owedge B)_{ijkl}=A_{ik}B_{jl}-A_{il}B_{jk}-A_{jk}B_{il} +A_{jl}B_{ik}.$$
We will occasionally use that $g\owedge g=2I$.

In \cite{BrendleSchoen2009}, the following  family of cones $\hat C(s)$ was introduced 
by Brendle and Schoen in order to prove the Differentiable Sphere Theorem, in the spirit of the work of B\"ohm and Wilking \cite{BohmWilking2008}.

\begin{defn}
Define a 
family of closed, convex, $O(n)$-invariant cones $\hat C(s)$, $s\in (0,+\infty)$, as follows:
\begin{enumerate}
\item[ (a)]
for $s\in (0,\frac12]$,
\begin{equation}
\hat C(s)=\left\{\ell_{a,s}(R): R\in  \CPICt,\; \Ric(R)\geq \frac{p}{n}\scal(R) \right\},
\end{equation}
where 
$$2a=\frac{2s+(n-2)s^2}{1+(n-2)s^2},\quad p=1-\frac{1}{1+(n-2)s^2}.$$
\item[ (b)] for $s\in (\frac12,+\infty)$,
\begin{equation}
\hat C(s)=\left\{ \ell_{s,\frac12}(R): R\in  \CPICt,\; \Ric(R)\geq \frac{p}{n}\scal(R)\right\},
\end{equation}
where 
$$ p=1-\frac{ 4}{n-2+8s}.$$
\end{enumerate}
\end{defn}

The importance of the cones $\hat C(s)$ stems from the following 
proposition, which can be used to show that certain Ricci flows become round.

\begin{prop}[Propositions 13, 14 and 15 in \cite{BrendleSchoen2009}, cf. \cite{BohmWilking2008}]
\label{properties-cone}
For every $s>0$, the cone $\hat C(s)$ satisfies the properties {\bf (I)-(IV)}.
Moreover, $\hat C(s)$ varies continuously in $s$.
\end{prop}

The continuity here is with respect to the Hausdorff topology once we intersect with the closed unit ball $\{R\in\mathcal{C}_B(\mathbb{R}^n)\ :\ |R|\leq 1\}$.
As we are dealing with \emph{convex} cones, the complements of the cones also vary continuously.

As $s\to +\infty$, the cone $\hat C(s)$  converges to  $\mathbb{R}_{\geq 0} I$. 
In particular, knowing that there exists $s_0>0$ such that
$R \in \hat C(s)$ for all $s\geq s_0$ is equivalent to knowing that $R$ is the curvature tensor of a space-form with non-negative curvature. 
It was shown by Brendle-Huisken-Sinestrari
\cite{BrendleHuiskenSinestrari2011} that 
compact ancient solutions of Ricci flow for $t\in (-\infty,0)$ with $R_{g(t)}\in \hat C(s_0)$ for some $s_0>0$ and all $t< 0$ (on each tangent space) must have constant sectional curvature for each $t< 0$. 
This was later generalised by Yokota  \cite{TakumiYokota-CAG} to the case of complete ancient Ricci flows with possibly unbounded curvature. 

We follow the idea of \cite{BrendleHuiskenSinestrari2011} and define for each $s>0$, $ \check{C} (s)$ to be the cone of all algebraic curvature tensors $R\in \mathcal{C}_B(\mathbb{R}^n)$ satisfying
\begin{equation*}
\begin{split}
& R_{1313}+\lambda^2 R_{1414}+ \mu^2 R_{2323}+\lambda^2\mu^2 R_{2424}\\
&\quad -2\lambda\mu R_{1234}+\frac1s(1-\lambda^2)(1-\mu^2)\cdot \mathrm{scal}(R)\geq 0
\end{split}
\end{equation*}
for all orthonormal four-frames $\{e_1,e_2,e_3,e_4\}\subset \mathbb{R}^n$ and $\lambda,\mu\in [0,1]$.  Clearly, we have $\CPICt\subset  \check C(s)\subset \CPICo$ for all $s>0$. 
See Appendix \ref{geometric_interpretation_appendix} for a geometric interpretation of 
$\check C(s)$, which also implies these inclusions.
The cones $\check C(s)$ are nested, getting smaller as $s$ increases,
and  their intersection is $\CPICt$. Moreover, $\check C(s)$ converges to $\CPICt$ as $s\to\infty$.
Following the idea in the proof of \cite[Theorem 12]{BrendleHuiskenSinestrari2011}, 
in the spirit of \cite{BohmWilking2008},
we define 
$$\tilde C(b,s):=\ell_{b}(\check C(s)),$$
where we denote  
\begin{equation}
\label{ell_b}
\ell_{b}(\cone):=\ell_{a,b}(\cone)\quad\text{ for }\quad a:=b+\frac12(n-2)b^2.
\end{equation}

\begin{prop}\label{properties-cone-2}
For any $n\geq 4$, $s>0$, $b\in (0,\Upsilon_n)$, where 
\begin{equation}
\label{Ups_def}
\textstyle \Upsilon_n:=\frac{\sqrt{2n(n-2)+4}-2}{n(n-2)},
\end{equation}
the cone $\tilde C(b,s)$ satisfies properties {\bf (I)-(IV)}.
Moreover, for each such $b$, the cones $\tilde C(b,s)$ vary continuously in $s$.
\end{prop}

\begin{proof}
This is contained in the proof of \cite[Theorem 12]{BrendleHuiskenSinestrari2011}, which in turn appeals to 
\cite[Proposition 10]{BrendleHuiskenSinestrari2011} and \cite[Proposition 3.2]{BohmWilking2008}.
For property {\bf (III)} see also Lemma \ref{lower_R_in_PIC1}.
\end{proof}

\section{A priori estimates under pinching conditions}
\label{apriori_sect}

In this section, we will establish local a priori estimates along the Ricci flow.  
But first we give a result that says that $\ell_{a,b}$ turns positive curvatures into pinched curvatures.

\begin{lma}
\label{gen_lin_alg_lem}
Suppose, for $n\geq 3$, that 
$\cone\subset\mathcal{C}_B(\mathbb{R}^n)$ is a 
convex cone  with the properties that $\Ric(R)\geq 0$ for every $R\in \cone$ and that every $R\in \mathcal{C}_B(\mathbb{R}^n)$ of non-negative curvature operator lies in $\cone$.
Suppose further that $a>b\geq 0$.
Then there exists $\delta>0$ depending on $a,b,n$ and the cone $\cone$ such that if $S\in\cone$ and 
$R=\ell_{a,b}(S)$ then
$$R-\delta \,\mathrm{scal}(R) \cdot I \in \cone. $$
\end{lma}

\begin{proof}
By assumption, $S\in\cone$ and hence $\Ric(S)\geq 0$.
We claim that this implies that $\Ric(S)\owedge \idorg$ has non-negative curvature operator, in which case it must be in the cone $\cone$ by assumption.
Indeed, a computation tells us that 
for a general element $A=\alpha_{ij}e_i\otimes e_j\in \Lambda^2 \R^n$, we have
$$\Ric(S)\owedge g (A,A) = 4\alpha_{ij}\alpha_{kj} \Ric(S)_{ik},$$
and the right-hand side is non-negative for each $j$ separately.

Given any symmetric bilinear form $h$ on $\R^n$, we can compute that
$$\scal(h\owedge g)=2(n-1)\tr(h),$$
and hence
$$\scal(\Ric(S)\owedge g)=2(n-1)\mathrm{scal}(S)$$
and 
$$\mathrm{scal}(I)=n(n-1).$$
Unravelling the definition of $R$ gives
$$R=S+b\;\Ric(S)\owedge \idorg+\frac{2(a-b)}n\; \mathrm{scal}(S) \cdot  I,$$
so
$$\mathrm{scal}(R)=\left(2a(n-1)+1\right)\mathrm{scal}(S).$$
Therefore for $\delta\in (0,\frac{2(a-b)}{n\left(2a(n-1)+1\right)}]$, 
we obtain
\begin{equation*}
\begin{aligned}
R-\delta  \;\mathrm{scal} & (R)  \cdot I
=S+b\;\Ric(S)\owedge \idorg+\frac{2(a-b)}n\; \mathrm{scal}(S) \cdot  I
-\delta \;\mathrm{scal}(R) \cdot I\\
& = S+b\;\Ric(S)\owedge \idorg+\left(\frac{2(a-b)}{n}-\delta\left(2a(n-1)+1\right)\right)\;\mathrm{scal}(S) \cdot  I\\
&\in C
\end{aligned}
\end{equation*}
\end{proof}

We assemble some specific instances of the previous lemma in the following lemma.

\begin{lma}\label{lma:linear-algebra}
Suppose $R=\ell_{a,b}(S)$ for some $a>  b\geq 0$ and $S\in \mathcal{C}_B(\mathbb{R}^n)$, for $n\geq 3$.
Then there exists $\delta(n,a,b)>0$ such that the following holds:
\begin{enumerate}
\item[(i)] $\Ric(R)\geq \delta \;\mathrm{scal}(R) \idorg$ if $\Ric(S)\geq 0$;
\item[(ii)] $R-\delta \;\mathrm{scal}(R) \cdot I \in \CPICo$ if   $S\in \CPICo$;
\item[(iii)]$R-\delta \;\mathrm{scal}(R) \cdot I \in \CPICt$ if  $S\in \CPICt$;
\end{enumerate}
\end{lma}

It is also convenient to record a converse statement. 
\begin{lma}
\label{converse_lemma}
Suppose for $n\geq 3$ that $\cone$ is a closed convex cone in $\mathcal{C}_B(\mathbb{R}^n)$ such that
\begin{enumerate}
\item
$I$ is in the interior of the cone $\cone$, and 
\item
every $R\in\cone\setminus\{0\}$ has positive scalar curvature. 
\end{enumerate}
Then given 
$\e\in (0,\frac{1}{n(n-1)})$, for small enough $b>0$ depending on $\e$, $n$ and $\cone$ we have the following: 
Every $R\in \mathcal{C}_B(\mathbb{R}^n)$ 
with $R-\e\,\mathrm{scal}(R) \cdot I\in\cone$ satisfies 
$R\in \ell_{b}(\cone)$.
\end{lma}
Recall that $\ell_b$ is defined in \eqref{ell_b}.

\begin{proof}
Because $\e\in (0,\frac{1}{n(n-1)})$, by considering the scalar curvature of $R-\e\;\mathrm{scal}(R) \cdot I\in\cone$ and using that each element of $\cone$ has non-negative scalar curvature, we deduce that 
$\mathrm{scal}(R) \geq 0$.
Using  that every $\tilde R\in\cone\setminus\{0\}$ has \emph{positive} scalar curvature, and also that the cone is closed, we deduce that $|\tilde R|\leq c\,\mathrm{scal}(\tilde R)$ for every $\tilde R\in\cone$ and some $c>0$ depending only on $\cone$. 
This ensures that if $R$ is scaled to have $|R|=1$ then not only is the term $\e \,\mathrm{scal}(R) \cdot I$ in the cone, but so is a neighbourhood of radius $r_0>0$ depending on $\e$ and the cone, including on $c$, but not on the specific $R$ chosen with $|R|=1$.
Thus by convexity of the cone we can add an error $S$ of magnitude $|S|<r_0$
to $R$ and still have 
$$R+S=[R-\e \;\mathrm{scal}(R) \cdot I]+[S+\e \;\mathrm{scal}(R) \cdot I]\in\cone.$$
In particular, since $\ell_{b}(\cone)$ varies continuously in $b$, we conclude that for small enough $b>0$ 
we have $R\in \ell_{b}(\cone)$.
\end{proof}

The following lemma translates the pinching condition into a cone condition.

\begin{lma}\label{pinching-implies-cone}
Suppose $R\in \mathcal{C}_B(\mathbb{R}^n)$, $\e_0\in (0,\frac{1}{n(n-1)})$ and
\begin{equation}
\label{Rhyp}
R-\e_0 \; \mathrm{scal}(R)\cdot I \in \CPICo.
\end{equation}
Then $R\in \tilde C(b,s_0)$ for some $s_0(n,\e_0)>0$ and $0<b(n,\e_0)<\Upsilon_n$.
\end{lma}
Note that by tracing the hypothesis \eqref{Rhyp} and using that $\e_0\in (0,\frac{1}{n(n-1)})$ we see that
$\mathrm{scal}(R)\geq 0$ automatically in the lemma.

\begin{proof}
We apply Corollary \ref{PIC1_to_checkC} in the appendix to 
$S:=R-\e_0 \; \mathrm{scal}(R)\cdot I \in \CPICo$ in place of $R$, and with
$\e=\frac{\e_0}{2}$. 
We deduce that there exists $s_0>0$ depending only on $n$ and $\e_0$ such that 
$$S+\frac{\e_0}{2}\,\scal(S)\cdot I \in \check C(s_0).$$
Because $\mathrm{scal}(S)=(1-\e_0n(n-1))\mathrm{scal}(R)\leq \mathrm{scal}(R)$,
we can expand $S$ to give
$$R-\frac12\e_0\; \mathrm{scal}(R) \cdot I   \in \check C(s_0),$$
where we are using that $I$ lies in the convex cone $\check C(s_0)$.
In fact, $I$ lies in the \emph{interior} of the cone $\check C(s_0)$ 
(even of the smaller cone $\CPICt$)
and every $R\in \check C(s_0)\setminus\{0\}$ has positive scalar curvature. This latter fact follows because $\check C(s_0)\subset \CPICo$ and the same is true for $\CPICo$
e.g. by Lemma \ref{lower_R_in_PIC1}.
We can then invoke Lemma \ref{converse_lemma} to conclude that for some 
$b\in (0,\Upsilon_n)$ 
we have
$R\in \ell_{b}(\check C(s_0))=\tilde C(b,s_0)$.
\end{proof}

\medskip

We now prove a local version of Hamilton's celebrated ODE-PDE theorem. 

\begin{thm}[Local ODE-PDE theorem]
\label{localODEPDEthm}
Suppose for $n\geq 3$ that $\cone$ is a closed, convex, $O(n)$-invariant cone in $\mathcal{C}_B(\mathbb{R}^n)$ with the properties
\begin{enumerate} 
\item
For some $\e>0$, if $S\in \mathcal{C}_B(\mathbb{R}^n)$ satisfies 
$|S-I|\leq\e$, then $S\in \cone$.
\item
$\cone$ is invariant under the Hamilton ODE.
\end{enumerate}
Suppose further 
that $(M^n,g(t))$, $t\in [0,T]$, is a smooth solution to the Ricci flow with $g(0)=g_0$ such that for some $x_0\in M$, we have 
\begin{enumerate}
\item[(i)] $B_{g_0}(x_0,1)\Subset M$;
\item[(ii)] $R_{g_0}  \in \cone$ on $B_{g_0}(x_0,1)$;
\item[(iii)] $|R_{g(t)}|\leq c_0 t^{-1}$ on $B_{g_0}(x_0,1)\times (0,T]$ for some $c_0<\infty$.
\end{enumerate}
Then for every $l\geq 0$, there exists
$S_0(n,c_0,\e,l)>0$ such that for all $t\in [0,T\wedge S_0]$, we have
$$R_{g(t)}(x_0)+ t^l\cdot I \in \cone.$$
\end{thm}
\noindent Here we use the notation $a\wedge b=\min\{a,b\}$ for $a,b\in\mathbb{R}$.

\medskip
In this work we will apply the local ODE-PDE theorem in the case that $l=0$ and $\cone=\tilde C(b,s_0)$ for some $s_0>0$
and $b\in (0,\Upsilon_n)$, where $\Upsilon_n$ was defined in \eqref{Ups_def},
in order to show that the pinching condition is almost preserved locally.
In that case it can be viewed as being analogous to \cite[Lemma 3.1]{LeeTopping2022} in the three-dimensional theory. For ease of reference we record the consequence we require in the following lemma.

\begin{lma}\label{preserv-cone}
Suppose for $n\geq 4$ that $(M^n,g(t))$, $t\in [0,T]$, is a smooth solution to the Ricci flow with $g(0)=g_0$ such that for some $x_0\in M$, we have 
\begin{enumerate}
\item[(i)] $B_{g_0}(x_0,1)\Subset M$;
\item[(ii)] $R_{g_0}  \in \tilde C(b,s_0)$ on $B_{g_0}(x_0,1)$ for some $s_0>0$
and $b\in (0,\Upsilon_n)$; 
\item[(iii)] $|R_{g(t)}|\leq c_0 t^{-1}$ on $B_{g_0}(x_0,1)\times (0,T]$ for some $c_0<\infty$.
\end{enumerate}
Then there exists 
$S_0(n,c_0,b,s_0)>0$ such that for all $t\in [0,T\wedge S_0]$, 
$$R_{g(t)}(x_0)+I \in \tilde C(b,s_0).$$
\end{lma}

%

\begin{proof}[{Proof of Theorem \ref{localODEPDEthm}}]
We denote by $d$ the distance $d(R,S)=|R-S|$ on $\mathcal{C}_B(\mathbb{R}^n)$,
where the norm is the standard extension of the Euclidean norm on $\R^n$. The corresponding inner product could be written $\langle R,S\rangle=R_{ijkl}S_{ijkl}$.

Define a Lipschitz function $\varphi: B_{g_0}(x_0,1)\times [0,T]\to \R$ by
$\varphi(x,t)=d(R_{g(t)}(x),\cone)$.  We first show that $\varphi$ satisfies 
\begin{equation}
\label{varphi_evol2}
\rheat \varphi \leq C(n) |R_{g(t)}|\varphi
\end{equation}
in the (strong) barrier sense whenever $\varphi>0$.  
By this we mean that for each point $(x_1,t_1)\in B_{g_0}(x_0,1)\times (0,T)$ where $\varphi(x_1,t_1)>0$, we
can find a smooth function $\underline{\varphi}(x,t)$ defined in a space-time neighbourhood of $(x_1,t_1)$ such that 
$\varphi\geq \underline{\varphi}$ where defined, 
while at $(x_1,t_1)$ we have both
$\varphi(x_1,t_1)=\underline{\varphi}(x_1,t_1)$ and 
$$\lf(\frac{\p}{\p t}-\Delta_{g(t_1)}\ri) \underline{\varphi} \leq C(n) |R_{g(t_1)}|\underline{\varphi}.$$
(This notion of barrier sense is particularly well adapted to applications of the maximum principle.)

If this is the case, then it follows from \cite[Theorem 1.1]{LeeTam2020} that 
for any $l\geq 0$, there exists $S_0>0$ depending only on $n$, $c_0$ and $l$ such that
\begin{equation}
\label{var_conc2}
\varphi(x_0,t)\leq t^{c_0+2+l} \text{ if }t \in (0,S_0]
\end{equation}
since $\varphi(\cdot,0)=0$ 
on $B_{g_0}(x_0,1)$
and 
$\varphi\leq d(R_{g(t)}(x),0)\leq c_0t^{-1}$
on $B_{g_0}(x_0,1)\times (0,T]$.  
By shrinking $S_0>0$, depending now additionally on $\e$, we may assume $t^{c_0+2}<\e$, and hence 
\begin{equation}
\label{var_conc3}
\varphi(x_0,t)\leq \e t^l\quad\text{ if }t\in (0, S_0].
\end{equation}

By definition of $\varphi$ we have
$$R_{g(t)}(x_0)-\varphi(x_0,t) \xi(R_{g(t)}(x_0))\in \cone$$
whenever $R_{g(t)}(x_0)\notin \cone$, where 
for $S\notin \cone$ we write 
\begin{equation}\label{eqn:xi}
\xi(S)=(S-\pi(S))/|S-\pi(S)|,
\end{equation}
 where $\pi$ is the projection to the cone $\cone$.
In particular,  convexity of the cone implies that  if $R_{g(t)}(x_0)\notin \cone$ then
for $t\in (0,S_0]$ we have
\begin{equation*}
\begin{aligned}
R_{g(t)}(x_0)+ t^l I &=\big[R_{g(t)}(x_0) -\varphi(x_0,t) \xi(R_{g(t)}(x_0))\big] + \big[t^l I +\varphi(x_0,t) \xi(R_{g(t)}(x_0))\big]\\
&\in \cone,
\end{aligned}
\end{equation*}
by \eqref{var_conc3}
and the definition of $\e$ in the theorem.
This will complete the proof.

In order to prove  \eqref{varphi_evol2}
at an arbitrary point $(x_1,t_1)$ where $\varphi(x_1,t_1)>0$, we 
must locally trivialise the tangent bundle, and hence the bundle in which $R_{g(t)}$ lives, near $(x_1,t_1)$ so that we can compare $R_{g(t)}(x)$ and $R_{g(t_1)}(x_1)$
for $(x,t)$ near $(x_1,t_1)$.
At the fixed time $t_1$ we do this by radial parallel transport centred at $x_1$. In the time direction we adopt the time-dependent gauge transformation often referred to as the Uhlenbeck trick
\cite{CK04}. This 
allows us to compare tensors at different points. It
makes the fibre metric (and derived cones, and $\pi$ and $\xi$) constant in a neighbourhood of 
$(x_1,t_1)$, and leads to a simpler evolution equation 
\begin{equation}
\label{Revol_eq2}
\left(\frac{\partial}{\partial t}-\Delta\right) R=Q(R)
\end{equation}
for the curvature tensor, for $Q$ as in \eqref{Qdef}.

To show the evolution equation \eqref{varphi_evol2} in the barrier sense,  we work near our point $(x_1,t_1)$ where $\varphi(x_1,t_1)>0$,  or equivalently $d(R_{g(t_1)}(x_1),\cone)>0$.  Since $\varphi(x,t)$ is in general only continuous, we construct a barrier as follows.  Let 
$$\underline{\varphi}(x,t)=\langle \xi(R_{g(t_1)}(x_1)),R_{g(t)}(x)\rangle$$
for $(x,t)$ sufficiently close to $(x_1,t_1)$ so that the bundle trivialisation above 
is valid. Then
$\varphi(x_1,t_1)=\underline{\varphi}(x_1,t_1)$
and convexity of the cone implies $\underline{\varphi}(x,t)\leq \varphi(x,t)$ for $(x,t)$ in the neighbourhood of $(x_1,t_1)$ where $\underline{\varphi}$ is defined. 
Hence, $\underline{\varphi}$ serves as a lower barrier for $\varphi$ in the strongest possible sense. 

In order to compute the evolution equation for $\underline{\varphi}$ at 
$(x_1,t_1)$, we first note that we always have $|\pi(R)|\leq |R|$, and so writing $L_0$ for the Lipschitz constant of $Q$ on the unit ball in $\mathcal{C}_B(\mathbb{R}^n)$
(which depends only on $n$)
we can compute
\begin{equation*}
\begin{aligned}
|Q(R)-Q(\pi(R))| 
&\leq \textstyle
|R|^2 \cdot 
|Q(\frac{R}{|R|})-Q(\frac{\pi(R)}{|R|})|\\
&\leq \textstyle
|R|^2 L_0 
\left|\frac{R}{|R|}-\frac{\pi(R)}{|R|}\right|\\
&= L_0|R|\cdot |R-\pi(R)|.
\end{aligned}
\end{equation*}
At $(x_1,t_1)$ we then have 
\begin{equation}
\begin{split}
\lf(\frac{\p}{\p t}-\Delta_{g(t_1)}\ri)
\underline{\varphi}&=\langle \xi(R_{g(t_1)}(x_1)),Q(R_{g(t_1)}(x_1))\rangle\\
&\leq \langle \xi(R_{g(t_1)}(x_1)),Q(R_{g(t_1)}(x_1))-Q(\pi(R_{g(t_1)}(x_1)))\rangle\\
&\leq L_0 |R_{g(t_1)}| \varphi.
\end{split}
\end{equation}
Here we have used the fact that $\pi(R_{g(t_1)}(x_1))\in \partial \cone$ so that $$ \langle \xi(R_{g(t_1)}(x_1)),Q(\pi(R_{g(t_1)}(x_1)))\rangle \leq 0$$
because $\cone$ is invariant under the Hamilton ODE. 
We have also used that although $\xi(R_{g(t_1)}(x_1))$ is not parallel, its Laplacian vanishes at $(x_1,t_1)$ by construction.
This completes the proof.
\end{proof}

Next we will  show that for $R_{g(t)}$ inside the pinching cones $\hat C(s)$ or $\tilde C(b,s)$ of large magnitude, the pinching improves with time.
We use an approach of 
Brendle-Huisken-Sinestrari \cite{BrendleHuiskenSinestrari2011} developed in the global setting.

\medskip


\begin{lma}\label{improving-cone}
Suppose $\mathbf{C}$ is either 
$\hat C(s)$ or $\tilde C(b,s)$, for $s>0$ and $b\in (0,\Upsilon_n)$,
with $n\geq 4$.
Then  there exists $\hat T(n,\mathbf{C})>0$ 
such that the following holds: 
Suppose $(M,g(t)),t\in [0,T]$ is a smooth solution to the Ricci flow such that for some $x_0\in M$ and $r_0\in (0,1)$, we have 
\begin{enumerate}
\item [(i)] $B_{g(t)}(x_0,8(n-1)r_0^{-1})\Subset M$;
\item[(ii)] $R_{g(t)}+I \in \mathbf{C}$ on $B_{g(t)}(x_0,8(n-1)r_0^{-1})$;
\item[(iii)] $\Ric_{g(t)}\leq  (n-1)r_0^{-2}$ on $B_{g(t)}(x_0,r_0)$,
\end{enumerate}
for all $t\in [0,T]$. 
 Then for all $t\in [0,T\wedge \hat T]$,
$$R_{g(t)}(x_0)+\left(4-t\cdot \mathrm{scal}(R_{g(t)}(x_0)) \right)I \in \mathbf{C}.$$
\end{lma}

\medskip

\begin{proof}
For each $t\in [0,T]$ and $x\in B_{g(t)}(x_0,8(n-1)r_0^{-1})$,
we denote 
\begin{equation}
\begin{aligned}
S=S_{g(t)}(x) &:=R_{g(t)}(x)+(2-t\,n(n-1)-t\; \mathrm{scal}(R_{g(t)}(x)))\cdot I\\
&=(R_{g(t)}(x)+I)+(1-t\; \mathrm{scal}(R_{g(t)}(x)+I))\cdot I\\
& =\hat R + (1-t\; \mathrm{scal}(\hat R))\cdot I, \label{simpler_S}
\end{aligned}
\end{equation}
where we abbreviate $\hat R := R_{g(t)}(x)+I$,
and define 
$$\varphi(x,t)=d \left( \mathbf{C}, S_{g(t)}(x)\right).$$
As we shall see at the end, constraining $\varphi$ not to grow too much will imply the conclusion of the lemma.

Observe that by convexity of the cone $\mathbf{C}$, and the facts that $R_{g(0)}+I\in \mathbf{C}$ 
on $B_{g(0)}(x_0,8(n-1)r_0^{-1})$
and $I\in\mathbf{C}$, we have
$S_{g(0)}=(R_{g(0)}+I)+I\in \mathbf{C}$, and hence $\varphi(\cdot,0)\equiv 0$, throughout $B_{g(0)}(x_0,8(n-1)r_0^{-1})$.

We first show that there exist $\hat T, \delta>0$ 
depending only on $n$ and $\mathbf{C}$ such that $\varphi$ satisfies 
$$\rheat \varphi \leq -\delta t^{-2}\varphi^2$$
in the  barrier sense 
at every point $(x_1,t_1)$ for which $t_1\in (0,T\wedge \hat T]$,
$x_1\in B_{g(t_1)}(x_0,8(n-1)r_0^{-1})$ and $\varphi(x_1,t_1)>0$.
We want to construct a lower barrier for $\varphi$ at $(x_1,t_1)$ 
satisfying
the desired evolution inequality.

Borrowing ideas and notation from the proof of 
Theorem \ref{localODEPDEthm},
including the bundle trivialisation near to $(x_1,t_1)$,
we define 
$$\underline{\varphi}(x,t)=\langle \xi(S_{g(t_1)}(x_1)),S_{g(t)}(x)\rangle$$
where $\xi$ is given by \eqref{eqn:xi}. This satisfies $\underline{\varphi}(x,t)\leq \varphi(x,t)$ for all $(x,t)$ in 
the neighbourhood of $(x_1,t_1)$ where the bundle trivialisation holds, 
with equality at $(x_1,t_1)$. 

Recall the evolution equation \eqref{Revol_eq2} for $R_{g(t)}(x)$, 
and the evolution equation 
$$\rheat\scal(R_{g(t)})
=2|\Ric(R_{g(t)})|^2$$
for $\scal(R_{g(t)})$, \cite[Proposition 2.5.4]{RFnotes},
which give us
$$\left(\frac{\partial}{\partial t}-\Delta\right) S = Q(R)-\scal(R)I-2t|\Ric(R)|^2I  -n(n-1)I.$$
At $(x_1,t_1)$ we thus obtain
\begin{equation}
\label{custard}
\rheat \underline{\varphi}
=\langle \xi(S), Q(R)-\mathrm{scal}(R)I 
-2t|\Ric(R)|^2 I   -n(n-1)I  \rangle.
\end{equation}
We will estimate the right-hand side by adapting the approach of 
\cite[Lemma 5]{BrendleHuiskenSinestrari2011} and \cite[Lemma 10]{TakumiYokota-CAG}.
A useful heuristic for the proof is that by constraining $\hat T>0$ (and hence $t=t_1$) to be small, we will show that $|R|$ is large and hence that the dominant term on the right-hand side of \eqref{custard} is the quadratic term $\langle \xi(S), Q(R)\rangle$. 
This term will be comparable to $\langle \xi(S), Q(\pi(S))\rangle$ because 
we will show that $R$ and $\pi(S)$ are close, relative to their magnitude.
To control this we follow \cite{TakumiYokota-CAG}, cf. \cite{BrendleHuiskenSinestrari2011}, and use property ${\bf (IV)}$ to obtain the estimate 
\begin{equation}
\label{BHS_est}
\langle \xi(S), Q(\pi(S))\rangle \leq -3\mu |\pi(S)|^2
\end{equation}
for some $\mu>0$ depending on $n$ and the cone. The right-hand side will be negative enough to absorb all the error terms and conclude the lemma.
We now give the details.

Because $\hat{R}\in\mathbf{C}$ and $I\in \mathbf{C}$, but $S=\hat R + (1-t\; \mathrm{scal}(\hat R))\cdot I\notin\mathbf{C}$, 
we must have $t\,\scal(\hat{R})>1$ at $(x_1,t_1)$.
But $\scal(\hat{R})\leq \alpha |\hat R|$, where we will use $\alpha$ to represent a $n$-dependent constant that is allowed to increase with each use. In particular,  
$|\hat R|\geq \frac{1}{\alpha t}$.
By reducing $\hat T>0$ if necessary, this forces $|\hat R|\geq 2|I|$, which implies
that $R=\hat R-I$ has comparable magnitude in that
$$\frac12 |R|\leq |\hat R|\leq 2|R|.$$
In particular, 
\begin{equation}
\label{IR_simple}
|I|\leq |R|.
\end{equation}
We compare $R$ and $\pi(S)$ by estimating
\begin{equation}
\label{RpiS}
\begin{aligned}
|R-\pi(S)| &= |\hat R -\pi(S) -I|\\
&\leq |\hat R -\pi(S)|+|I|\\
&\leq |\hat R -S|+|I|\\
&=(t\,\scal(\hat{R})-1)|I|+|I|\\
&=t\,\scal(\hat{R})|I|.
\end{aligned}
\end{equation}
A first consequence is that $|R-\pi(S)|\leq \alpha t |\hat R|\leq 2\alpha t |R|$, so by reducing $\hat T>0$ if necessary, we may assume  that
$$|\pi(S)|\geq |R|-|R-\pi(S)|\geq (1-2\alpha t)|R|\geq \frac12 |R|.$$
Therefore $\scal(\hat{R})\leq \alpha |\pi(S)|$ and $|\scal(R)|\leq \alpha |\pi(S)|$.
In particular, \eqref{RpiS} develops to
\begin{equation}
\label{liberty}
|R-\pi(S)| \leq \alpha t |\pi(S)|.
\end{equation}
We can use this estimate to control the first term on the right-hand side of 
\eqref{custard}.
By the quadratic nature of $Q$, 
\begin{equation*}
\begin{aligned}
Q(R) &=Q(\pi(S)+(R-\pi(S)))\\
&= Q(\pi(S))+Q(R-\pi(S)) +\pi(S) * (R-\pi(S)),
\end{aligned}
\end{equation*}
where the $*$ notation is explained in \cite[\S 2.1]{RFnotes}.
In particular, 
\begin{equation*}
\begin{aligned}
\langle \xi(S), Q(R)\rangle &\leq 
\langle \xi(S), Q(\pi(S))\rangle + \alpha |R-\pi(S)|^2 +\alpha |\pi(S)|\cdot |R-\pi(S)|\\
&\leq -3\mu |\pi(S)|^2 + \alpha (t^2+t) |\pi(S)|^2\\
&\leq -2\mu |\pi(S)|^2,
\end{aligned}
\end{equation*}
by \eqref{BHS_est} and \eqref{liberty}, where (as usual) $\alpha$ is a $n$-dependent constant that can increase with each use, and we may have had to reduce $\hat T>0$.
The second term on the right-hand side of 
\eqref{custard} can be controlled by 
$$\langle \xi(S), -\mathrm{scal}(R)I \rangle \leq
|\mathrm{scal}(R)|\cdot |I| \leq
\alpha |\pi(S)| \leq (t\,\scal(\hat{R}))\alpha |\pi(S)|
\leq t \alpha |\pi(S)|^2.$$
Similarly, the fourth 
term on the right-hand side of 
\eqref{custard} can be controlled by
$$\langle \xi(S), -n(n-1)I\rangle \leq \alpha |I| \leq \alpha |R| \leq 
(t\,\scal(\hat{R}))\alpha |\pi(S)|\leq t \alpha |\pi(S)|^2,$$
where we have used \eqref{IR_simple}.
Finally, the third term on the right-hand side of 
\eqref{custard} can be controlled by
$$\langle \xi(S), -2t|\Ric(R)|^2 I \rangle \leq \alpha t|R|^2\leq t \alpha |\pi(S)|^2.$$
Combining, \eqref{custard} gives
\begin{equation}
\begin{aligned}
\rheat \underline{\varphi} &\leq -2\mu |\pi(S)|^2 + t \alpha |\pi(S)|^2\\
&\leq -\mu |\pi(S)|^2,
\end{aligned}
\end{equation}
after possibly reducing $\hat T>0$ again.
Meanwhile, because $\hat R\in \mathbf{C}$, we compute
\begin{equation}
\begin{aligned}
\varphi(x_1,t_1) &\leq |\hat{R}-S|=(t\,\scal(\hat{R})-1)|I|\\
&\leq t\, \scal(\hat{R})|I| \leq \alpha t\, |\pi(S)|.
\end{aligned}
\end{equation}
Thus there exists $\delta>0$ depending only on $n$ and $\mathbf{C}$
such that 
\begin{equation}
\rheat \underline{\varphi}\leq -\delta t^{-2} \varphi^2.
\end{equation}
at $(x_1,t_1)$ as claimed.

Now we are ready to set up a maximum principle argument. By hypothesis (iii) and 
\cite[Lemma 8.3(a)]{Perelman2002},
the function $\eta(x,t)=d_{g(t)}(x,x_0)+ \frac{5}{3}(n-1) r_0^{-1}t$ satisfies 
$$\rheat \eta \geq 0$$
in the barrier sense
whenever $d_{g(t)}(x,x_0)\geq r_0$. Let $\phi$ be a smooth non-increasing function on $\mathbb{R}$ such that $\phi\equiv 1$ on $(-\infty,1]$, $\phi$ vanishes outside $(-\infty,2]$ and satisfies $\phi''\geq -10^2\phi$, $|\phi'|\leq 10^2\phi^{1/2}$. Define 
$$ \Phi(x,t)=e^{-10^2tA^{-2}r_0^2}\phi\left(\frac{\eta(x,t)}{Ar_0^{-1}} \right),\;\; A=4(n-1)$$
so that $(\partial_t-\Delta_{g(t)})\Phi \leq 0$ in the sense of barriers. Moreover due to the choice of $A$, $\Phi(x_0,t)= e^{-10^2tA^{-2}r_0^2}$ for $t\in[0,T\wedge \hat T]$ and  $\Phi$ vanishes outside $B_{g(t)}(x_0,2Ar_0^{-1})$, i.e., $B_{g(t)}(x_0,8(n-1)r_0^{-1})$.

\medskip

Consider the function $G=\Phi\varphi$. We fix $s\in (0,T\wedge \hat T]$. Considering the support of $\Phi$, we see that $G$ attains its maximum on $M\times [0,s]$ at some point $(x_1,t_1)$ with $x_1\in B_{g(t_1)}(x_0,8(n-1)r_0^{-1})$. 
Because we showed that 
$\varphi(\cdot,0)\equiv 0$, throughout $B_{g(0)}(x_0,8(n-1)r_0^{-1})$, we see that $G(\cdot,0)\equiv 0$ throughout $M$.
Since our goal is to obtain an upper bound for $G$,
we may assume $t_1>0$ and $G(x_1,t_1)>0$; in particular we have both $\varphi(x_1,t_1)>0$ (so our calculations above are  valid) and 
$\Phi(x_1,t_1)>0$. 
We may assume $\Phi$ and $\varphi$ to be smooth when we apply the maximum principle; for example see \cite[Section 7]{SimonTopping2016} for a detailed exposition.

In this case, at $(x_1,t_1)$ we have $ \Phi \nabla \varphi =-\varphi\nabla \Phi$ from $\nabla G=0$ and hence,
\begin{equation}
\begin{split}
0\leq \rheat G&=  \varphi  \cdot  \rheat \Phi  + \Phi \cdot \rheat \varphi \\
&\quad -2\langle \nabla \varphi,\nabla\Phi\rangle \\
&\leq -\delta t^{-2} \varphi^2 \Phi + \frac{2|\nabla \Phi|^2}{\Phi} \varphi \\
&\leq \varphi\left(-\delta t^{-2}  G + \frac{2\cdot 10^4}{A^2r_0^{-2}} \right).
\end{split}
\end{equation}
Thus, $\sup_{M\times [0,s]}G=G(x_1,t_1)
\leq C_n r_0^2\delta^{-1}s^2$ for some dimensional constant $C_n>0$. Since $s$ is arbitrary in $[0,T\wedge \hat T]$, we have 
$$G(x,t)
\leq C_n r_0^2\delta^{-1}t^2$$
for all $(x,t)\in M\times [0,T\wedge\hat T]$.
Since $r_0\in (0,1)$, evaluating at $x_0$, where $\Phi(x_0,t)=e^{-10^2tA^{-2}r_0^2}$,  yields
\begin{equation}
\label{varphi_estimate}
\varphi(x_0,t)\leq e^{10^2tA^{-2}r_0^2}\cdot C_nr_0^2 \delta^{-1}t^2 \leq Lt^2,
\end{equation}
for some constant $L<\infty$ depending only on $n$ and $\mathbf{C}$, for $t\in [0,T\wedge \hat T]$.
By definition, we have $S-\varphi(x_0,t)\xi(S)\in \mathbf{C}$, and so
we can compute at $x_0$
\begin{equation}
\begin{aligned}
R+(4-t\,\scal(R))I &=S+2I+n(n-1)tI\\
& =[S-\varphi(x_0,t)\xi(S)]+[2I+n(n-1)tI+\varphi(x_0,t)\xi(S)]
\end{aligned}
\end{equation}
which must lie in $\mathbf{C}$ for $t\in [0,T\wedge\hat T]$,
provided we constrain $\hat{T}>0$ depending on $n$ and $\mathbf{C}$,
where we have used the convexity of the cone $\mathbf{C}$
and the fact that $I$ lies in the interior of $\mathbf{C}$, along with the estimate
\eqref{varphi_estimate}.
\end{proof}

\medskip

We now use the cone almost improving nature (i.e. Lemma~\ref{improving-cone}) to 
obtain  curvature estimates.

\begin{prop}\label{cur-estimate}
Suppose $M^n$ is a non-compact (connected) manifold for $n\geq 4$, and $g(t)$, $t\in [0,T]$ is a smooth solution to the Ricci flow on $M$ so that for some $x_0\in M$, $s_0>0$ and $b\in (0,\Upsilon_n)$,
we have 
\begin{enumerate}
\item[(i)] $B_{g(t)}(x_0,1)\Subset M$ for all $t\in [0,T]$; 
\item[(ii)] $R_{g(t)}+I \in \tilde C(b,s_0)$ on $B_{g(t)}(x_0,1)$ for $t\in [0,T]$.
\end{enumerate}
Then there exist $C_0(n,b,s_0),S_1(n,b,s_0)>0$ such that for $t\in (0,S_1\wedge T]$, 
$$|R_{g(t)}(x_0)|\leq C_0 t^{-1}.$$
\end{prop}

\medskip
Thus we obtain $C_0/t$ decay of the full curvature tensor analogous to the estimate 
in \cite[Lemma 3.3]{LeeTopping2022}.

\begin{proof}
The proof of the curvature estimate uses the Perelman-inspired point-picking argument from \cite[Lemma 2.1]{SimonTopping2016}, and initially mirrors the proof of 
\cite[Lemma 3.3]{LeeTopping2022}.

Suppose the conclusion is false for some
$n\geq 4$, $b\in (0,\Upsilon_n)$ 
and $s_0>0$.
Then 
for any $a_k\to +\infty$, we can find a sequence of non-compact manifolds $M_k^n$, Ricci flows $g_k(t),t\in [0,T_k]$ and $x_k\in M_k$ satisfying the hypotheses but so the curvature estimate fails with $C_0=a_k$ in an arbitrarily short time. We may assume $a_kT_k\to 0$.   By smoothness of each Ricci flow, we can choose $t_k\in (0,T_k]$ so that
\begin{enumerate}
\item[(i)] $B_{g_k(t)}(x_k,1)\Subset M_k$ for $t\in [0,t_k]$;
\item[(ii)] $R_{g_k(t)}+I\in \tilde C(b,s_0)$ on $B_{g_k(t)}(x_k,1)$ for $t\in [0,t_k]$;
\item[(iii)]$|R_{g_k(t)}(x_k)|<a_k t^{-1}$ 
for $t\in (0,t_k)$;
\item[(iv)]$|R_{g_k(t_k)}(x_k)|=a_k t_k^{-1}$.
\end{enumerate}
\medskip

By (iv) and the fact that $a_kt_k\rightarrow 0$, \cite[Lemma 5.1]{SimonTopping2016} implies that for sufficiently large $k$, we can find $\b(n)>0$, times $\tilde t_k\in (0,t_k]$ and points $\tilde x_k\in B_{g_k(\tilde t_k)}(x_k,\frac{3}{4}-\frac{1}{2}\b\sqrt{a_k \tilde t_k})$ such that 
\begin{equation}\label{CURVATURE}
|R_{g_k(t)}(x)|\leq 4|R_{g_k(\tilde t_k)}(\tilde x_k)|=4Q_k
\end{equation}
whenever $d_{g_k(\tilde t_k)}(x,\tilde x_k)<\frac{1}{8}\b a_kQ_k^{-1/2}$ and $\tilde t_k-\frac{1}{8}a_kQ_k^{-1}\leq t\leq \tilde t_k$ where $\tilde t_k Q_k\geq a_k\rightarrow +\infty$.

Included in the proof of \cite[Lemma 5.1]{SimonTopping2016}
is that for each such $(x,t)$ we have $x\in B_{g_k(t)}(x_k,1)$
i.e. the cylinder 
$B_{g_k(\tilde t_k)}(\tilde x_k,\frac{1}{8}\b a_kQ_k^{-1/2})\times
[\tilde t_k-\frac{1}{8}a_kQ_k^{-1},\tilde t_k]$
where \eqref{CURVATURE} holds lies within the region where (ii) holds. 

\medskip

Consider the parabolic rescaling centred at $(\tilde x_k,\tilde t_k)$, namely $\tilde g_k(t)=Q_k g_k(\tilde t_k+Q_k^{-1}t)$ for $t\in [-\frac{1}{8}a_k,0]$ so that
\begin{enumerate}
\item[(a)] $|R_{\tilde g_k(0)}(\tilde x_k)|=1$;
\item[(b)] $|R_{\tilde g_k(t)}|\leq 4$ on $B_{\tilde g_k(0)}(\tilde x_k,\frac18 \b a_k)\times [-\frac18 a_k,0]$, and
\item[(c)]  $R_{\tilde g_k(t)}+Q_k^{-1}I \in \tilde C(b,s_0)$ on $B_{\tilde g_k(0)}(\tilde x_k,\frac18 \b a_k)\times [-\frac18 a_k,0]$.
\end{enumerate}
\medskip
If we had a uniform positive lower bound on the injectivity radii
$\mathrm{inj}(\tilde g_k(0))(\tilde x_k)$, then Hamilton's compactness
theorem  would enable us to extract a  subsequence converging in the $C^\infty$ Cheeger-Gromov sense to a complete ancient solution of Ricci flow $g_\infty(t)$ which would be non-compact, non-flat and have bounded curvature. Moreover, we would have $R_{g_\infty(t)}\in \tilde C(b,s_0)$ for all $t\leq 0$ so that \cite[Lemma 15]{TakumiYokota-CAG} would 
apply to give a contradiction to the non-compactness of the underlying manifolds $M_k$. 
To accommodate the lack of an injectivity radius lower bound, we can instead take a local limit $\tilde g_\infty(t)$
of $\tilde g_k(t)$ by lifting to a Euclidean ball via the exponential map of $\tilde g_k(0)$ 
as in the proof of \cite[Lemma 3.3]{LeeTopping2022}.
(We only need to consider the local limit $\tilde g_\infty(0)$, but we take the limit at each time for consistency with \cite{LeeTopping2022}.)
Since the limit is a priori only locally defined on a ball in Euclidean space, we need to extract more information along the sequence first.

We now improve the pinching behaviour at $t=0$ by using the fact that the flow is \textit{almost} a complete ancient solution with bounded curvature. We first show that it becomes almost PIC2
pinched.

\begin{claim}\label{claim:pinching-1}
For any $L,\e>0$ and $s\geq s_0$, there exists $N\in \mathbb{N}$ such that for all $k>N$, we have 
$$R_{\tilde g_k(t)}+\e I \in \tilde C(b,s)$$
on $B_{\tilde g_k(0)}(\tilde x_k,L)\times [-L^2,0]$.
\end{claim}

\begin{proof}[Proof of claim]
Let $\mathcal{S}$ be the set of $ s'\in [s_0,+\infty)$ so that for all $L>0$ and $\e\in (0,1)$, we can find $N\in \mathbb{N}$ such that for all $k>N$, we have 
\begin{equation}
\label{starstar}
R_{\tilde g_k(t)}+\e I \in  \tilde C(b,s')
\end{equation}
on $B_{\tilde g_k(0)}(\tilde x_k,L)\times [-L^2,0]$. Clearly, $s_0\in \mathcal{S}$
by (c) above. We want to show that $\mathcal{S}$ is open and closed so that $\mathcal{S}=[s_0,+\infty)$.

We first show that $\mathcal{S}$ is closed.
Let $s_i\in \mathcal{S}$ so that $s_i\to s_\infty\in \mathbb{R}$. 
By definition of the cones $\tilde C(b,\tau)$ we see we have the following type of continuity of the cones with respect to $\tau$: 
For any $\e'>0$, there exists $\delta>0$ such that if $R\in \tilde C(b,s)$ for some $|s-s_\infty|<\delta$, then $R+\e'\;  \mathrm{scal}(R)\cdot I\in \tilde C(b,s_\infty)$.

For any given $L>0$ and $\e\in (0,1)$
for which we want \eqref{starstar} to be true with $s'=s_\infty$ for sufficiently large $k$,
we choose 
$\e'=\frac1{8n(n-1)} \e$
and obtain a corresponding $\delta>0$ from the above continuity. We fix $i_0$ sufficiently large so that $|s_{i_0}-s_\infty|<\delta$. Since $s_{i_0}\in \mathcal{S}$, 
we can apply the definition of $s_{i_0}\in \mathcal{S}$ with $\e$ replaced by $\frac{\e}{4}$ to deduce that 
there exists $N\in \mathbb{N}$ 
such that for all $k>N$,
$$\textstyle R_{\tilde g_k(t)}+\frac{\e}{4} I \in  \tilde C(b,s_{i_0}) $$
on $B_{\tilde g_k(0)}(\tilde x_k,L)\times [-L^2,0]$. 
Hence, $\hat R=R_{\tilde g_k(t)}+\frac{\e}{4} I$ satisfies 
\begin{equation}
\begin{split}
\textstyle
\hat R+\frac1{8n(n-1)} \e \cdot \mathrm{scal}(\hat R) I \in  \tilde C(b,s_\infty ).
\end{split}
\end{equation}
%
Since
$$\textstyle \mathrm{scal}(\hat R)=\mathrm{scal}(R_{\tilde g_k(t)})+ \frac{\e}{4} n(n-1),$$
we can then unwind the definition of $\hat{R}$ to give
\begin{equation}
R_{\tilde g_k(t)}+ \left(\frac{\e}{4} +\frac18 \cdot \frac{\e^2}{4}  +\frac1{8n(n-1)} \e \cdot  \mathrm{scal}(R_{\tilde g_k(t)}) \right)I \in \tilde C(b,s_\infty).
\end{equation}
A coarse consequence of the curvature bound $|R_{\tilde g_k(t)}|\leq 4$ from (b)
is that $\scal(R_{\tilde g_k(t)})\leq 4n(n-1)$, so 
keeping in mind that $\e\leq 1$ 
we find that
$$R_{\tilde g_k(t)}+\e I \in  \tilde C(b,s_\infty)$$
on $B_{\tilde g_k(0)}(\tilde x_k,L)\times [-L^2,0]$ for sufficiently large $k$. 
Hence $\mathcal{S}$ is closed.

\bigskip

It remains to show that $\mathcal{S}$ is open. 
Let $s'\in \mathcal{S}$.
It suffices to show that for some $\sigma'>0$, the following is true:
For all $L>0$ and $\e\in (0,1)$, the modified curvature tensor
$\check R=R_{\tilde g_k(t)}+\e I$ satisfies $\check R-\sigma'\cdot \mathrm{scal}(\check R)I \in \tilde C(b,s')$ 
on $B_{\tilde g_k(0)}(\tilde x_k,L)\times [-L^2,0]$
for sufficiently large $k$. The result will then follow using the continuity of the cones.

Set $\e'=\frac{\e}{32}$
and let $L'>L$ be a constant to be chosen later. Since $s'\in \mathcal{S}$, there exists $N\in \mathbb{N}$ such that for all $k>N$, we have 
$$R_{\tilde g_k(t)}+\e' I \in  \tilde C(b,s') $$
on $B_{\tilde g_k(0)}(\tilde x_k,L')\times [-(L')^2,0]$. 

Define $\sigma':=\min\{\hat T(n, \tilde C(b,s') ),\frac1{2n(n-1)}\}$, 
where $\hat T$ is the positive constant obtained from Lemma~\ref{improving-cone} with $\mathbf{C}$ chosen to be $\tilde C(b,s')$.
Set $r_0=\frac18\sqrt{\e}\in (0,1)$,
let $(x,\tau)\in B_{\tilde g_k(0)}(\tilde x_k,L)\times [-L^2,0]\Subset B_{\tilde g_k(0)}(\tilde x_k,L')\times [-(L')^2,0]$ and define a rescaled Ricci flow
$$\hat g_k(t)=4r_0^2\, \tilde g_k\left(\tau+ \frac{t-\sigma'}{4r_0^2} \right), \; t\in [0,\sigma'].$$

In order to ensure that $\hat g_k(t)$ is defined on the whole time interval $[0,\sigma']$, even if $\tau$ is chosen at its least value $-L^2$, we ask that $L'$ is large enough so that $L^2+\frac{\sigma'}{4r_0^2}\leq (L')^2$.
We would also like to choose $L'$ sufficiently large so that for every $t\in [0,\sigma']$ there is enough space to fit balls $B_{\hat g_k(t)}(x, 8(n-1)r_0^{-1})$ within 
$B_{\tilde g_k(0)}(\tilde x_k,L')$.
By the overall curvature bound of part (b), we have 
$|R_{\tilde{g}_k(\cdot)}|\leq 4$, 
and this bound is enough to allow us to compare geodesic balls 
at different times.
In particular, we have the inclusion
$$B_{\hat g_k(t)}(x, 8(n-1)r_0^{-1}) \subset
B_{\tilde g_k(0)}(x,L_0)
$$
for all $t\in [0,\sigma']$,
for some $L_0$ depending on $n$, $\e$, $L$ and $\sigma'$. 
Thus it suffices to insist that $L'\geq L+L_0$. 
Furthermore since $a_k\to+\infty$, for sufficiently large $k$ we have 
\begin{equation}
B_{\tilde g_k(0)}(\tilde x_k,L')=B_{g_k(\tilde t_k)}(\tilde x_k,L'Q_k^{-1/2})\subset B_{g_k(\tilde t_k)}(\tilde x_k,\frac18 \beta a_kQ_k^{-1/2})\subset 
B_{g_k(\tilde t_k)}(x_k,1)
\end{equation}
so that the region under consideration is compactly contained in $M_k$.

We conclude that, increasing $N$ if necessary (possibly depending also on $\e$, $L'$), for $k>N$,
$\hat g_k(t),t\in [0,\sigma']$ is a Ricci flow such that for all $t\in [0,\sigma']$,
\begin{enumerate}
\item $B_{\hat g_k(t)}(x, 8(n-1)r_0^{-1})\Subset M_k$;
\item $\Rold_{\hat g_k(t)}+
I \in  \tilde C(b,s' )$ on $B_{\hat g_k(t)}(x, 8(n-1)r_0^{-1})$;
\item $\Ric_{\hat g_k(t)}\leq 
(n-1) r_0^{-2}$ on $B_{\hat g_k(t)}(x,r_0)$.
\end{enumerate}

\medskip

\noindent
Applying Lemma~\ref{improving-cone} to $\hat{g}_k(t)$ at $(x,\sigma')$ 
gives
$$R_{\hat{g}_k(\sigma')}(x)+\left(4-\sigma'\cdot \mathrm{scal}(R_{\hat{g}_k(\sigma')}(x)) \right)I \in \tilde C(b,s').$$
Rescaling the metric back by the same factor $4r_0^2=\frac{\e}{16}$ then gives
$$R_{\tilde g_k(\tau)}(x)+\left(\frac{\e}{4}
- \sigma'\cdot \mathrm{scal}(R_{\tilde  g_k(\tau)}(x))\right)I\in  \tilde C(b,s').$$
Rewriting using $\check R=R_{\tilde g_k(\tau)}(x)+\e I$ gives
\begin{equation}
\check R +\left(\frac{\e}{4}+\sigma' \e n(n-1)-\e\right)I- \sigma'\cdot \mathrm{scal}(\check R) I\in  \tilde C(b,s').
\end{equation}
By definition,  $\sigma'\leq \frac1{2n(n-1)}$, so the part in brackets is negative
and we deduce that for the given $\e,L>0$, 
\begin{equation}
\check R - \sigma'\cdot \mathrm{scal}(\check R) I\in  \tilde C(b,s')
\end{equation}
throughout $B_{\tilde g_k(0)}(\tilde x_k,L)\times [-L^2,0]$ for sufficiently large $k$,
as desired.
This shows that $\mathcal{S}$ is also open and hence $S=[s_0,+\infty)$. This completes the proof of the claim.
\end{proof}

Thanks to  Claim~\ref{claim:pinching-1}, any local limit 
$\tilde g_\infty(0)$ 
(in the sense described before Claim \ref{claim:pinching-1})
of the metrics $\tilde g_k(0)$
has curvature lying in 
$\tilde C(b,s):=\ell_{b}(\check C(s))$,
for all $s$ sufficiently large. By definition, being in $\check C(s)$ for all $s$ sufficiently large is equivalent to being in $\CPICt$. Thus we find that 
$\tilde g_\infty(0)\in \ell_{b}(\CPICt)\subset \CPICt$ because 
$\ell_{b}:=\ell_{a,b}$ for $a:=b+\frac12(n-2)b^2>b\geq 0$,
e.g. using Lemma \ref{lma:linear-algebra}.
Next, we will improve this control further by showing that 
$\tilde g_\infty(0)$ lies in $\hat C(s)$ for all $s$ sufficiently large, which will ultimately show that any local blow up limit 
$\tilde g_\infty(0)$ is in fact a space form.

\begin{claim}\label{claim:pinching-2}
There exists $\hat s_0>0$ such that the following is true: for any $L,\e>0$ and $s\geq \hat s_0$, there exists $N\in \mathbb{N}$ such that for all $k>N$, we have 
$$R_{\tilde g_k(t)}+\e I \in \hat C(s)$$
on $B_{\tilde g_k(0)}(\tilde x_k,L)\times [-L^2,0]$.
\end{claim}

\begin{proof}[Proof of claim]
Since $\check{C}(s)\to \CPICt$ as $s\to\infty$, 
we have $\tilde C(b,s)\to \ell_{b}(\CPICt)$  by definition. 
Lemma~\ref{lma:linear-algebra} tells us that if $R\in \ell_{b}( \CPICt)$ then
for $\delta>0$ depending on $b>0$ and $n$, we have $R-\delta\,\mathrm{scal}(R)\cdot I\in \CPICt$. Thus for $\tilde{s}>0$ large enough, 
$$\tilde C(b,s)\subset \{S\in \mathcal{C}_B(\mathbb{R}^n)\ :\ 
S-\frac{\delta}{2}\,\mathrm{scal}(S)\cdot I\in \CPICt\}$$
for every $s\geq \tilde{s}$, while for $\hat{s_0}>0$ small enough, by continuity of the cones $\hat C(s)$, and the fact that $\hat C(s)\to \CPICt$ as $s\searrow 0$, we have
$$\{S\in \mathcal{C}_B(\mathbb{R}^n)\ :\ 
S-\frac{\delta}{2}\,\mathrm{scal}(S)\cdot I\in \CPICt\}\subset \hat C(\hat s_0).$$
To summarise, for our given small $\hat s_0>0$, we have 
\begin{equation}
\label{initial_inclusion}
\tilde C(b,s)\subset \hat C(\hat s_0)
\end{equation}
for every $s\geq \tilde s$.

Mimicking the proof of Claim~\ref{claim:pinching-1}, we let $\mathcal{S}'$ be the set of $ s'\in [\hat s_0,+\infty)$ so that for all $L>0$, $\e\in (0,1)$, we can find $N\in \mathbb{N}$ such that for all $k>N$, we have 
$$R_{\tilde g_k(t)}+\e I \in \hat C(s')$$
on $B_{\tilde g_k(0)}(\tilde x_k,L)\times [-L^2,0]$. 
The set $\mathcal{S}'$ is non-empty because $\hat s_0\in \mathcal{S}'$
by Claim~\ref{claim:pinching-1} and \eqref{initial_inclusion}.
Now we can carry out the same argument as in the proof of Claim~\ref{claim:pinching-1} to show that $\mathcal{S}'=[\hat s_0,+\infty)$. This completes the proof of the claim.
\end{proof}

\medskip

By Claim~\ref{claim:pinching-2}, any local limit $\tilde g_\infty(0)$
of the metrics $\tilde g_k(0)$
is a space-form with non-negative curvature. 
As $|R_{\tilde g_k(0)}(\tilde x_k)|=1$, this forces $\Ric_{\tilde g_\infty(0)}=
\a$ for some dimensional constant $\a>0$. Now we can argue as in 
the proof of
\cite[Lemma 3.3]{LeeTopping2022} to draw a contradiction from the non-compactness of $M_k$ for $k$ sufficiently large. This finishes the proof.
\end{proof}

\section{Existence of Ricci flow under pinching}
\label{exist_sect}

In this section, we will construct a smooth complete Ricci flow solution with scaling invariant estimates from metrics with pinched curvature.  To do this, we will need a local existence theorem for the Ricci flow, which is one of the main goals of this section.

\begin{thm}\label{thm:local-RF}
For any $s_0>0$, $n\geq 4$ and $b\in (0,\Upsilon_n)$,
there exist $a_0(n,b,s_0)>0$ and $T(n,b,s_0)>0$ such that the following holds.  Suppose $(M^n,g_0)$ is a manifold such that 
\begin{enumerate}
\item[(a)] $B_{g_0}(p,4)\Subset M$;
\item[(b)] $R_{g_0} \in \tilde C(b,s_0)$ on $B_{g_0}(p,4)$;
\end{enumerate}
Then there exists a  smooth Ricci flow solution $g(t)$ defined on $B_{g_0}(p,1)\times [0,T]$ such that $g(0)=g_0$ and 
\begin{enumerate}
\item[(i)] $R_{g(t)}+I\in \tilde C(b,s_0)$; 
\item[(ii)]  $|R_{g(t)}|\leq a_0 t^{-1}$.
\end{enumerate}
\end{thm}
\medskip

We will need several ingredients for the proof of Theorem~\ref{thm:local-RF}. The first of these is a result of Hochard that allows us to construct a local Ricci flow on regions with bounded curvature by modifying an incomplete Riemannian metric at its extremities in order to make it complete, without increasing the curvature too much, and without changing it in the interior.

\begin{prop}[Proposition 4.2 in \cite{LeeTopping2022}, based on \cite{hochard_thesis}]\label{construction-local-RF-new}
For $n\geq 2$ there exist constants $\alpha\in (0,1]$ and $\Lambda>1$ depending on $n$ so that the following is true.
Suppose $(N^n,h_0)$ is a smooth manifold (not necessarily complete)
that satisfies $|R_{h_0}|\leq \rho^{-2}$ throughout, for some $\rho>0$.
Then there exists a smooth Ricci flow $h(t)$ on $N$
for $t\in [0,\alpha \rho^2]$, with the properties that 
\begin{enumerate}
\item[(i)] $h(0)=h_0$ on $N_\rho=\{x\in N: B_{h_0}(x,\rho)\Subset N\}$;
\item[(ii)] $|R_{h(t)}|\leq  \Lambda \rho^{-2}$ throughout $N\times [0,\alpha \rho^2]$.
\end{enumerate}
\end{prop}

We also recall the shrinking balls lemma, which is one of the local ball inclusion results  based on  the distance distortion estimates of Hamilton and Perelman from  \cite[Lemma 8.3]{Perelman2002}.
\begin{lma}[{\cite[Corollary 3.3]{SimonTopping2016}}]
\label{l-balls}
For $n\geq 2$ there exists a constant $\beta\geq 1$ depending only on $n$ such that the following is true. Suppose $(N^n,g(t))$ is a Ricci flow for $t\in [0,S]$ and $x_0\in N$   with $B_{g(0)}(x_0,r)\Subset N$ for some $r>0$, and $\Ric_{g(t)}\leq a/t$ on $B_{g_0}(x_0,r)$ for each $t\in (0,S]$. Then
$$B_{g(t)}\left(x_0,r-\beta\sqrt{a t}\right)\subset B_{g_0}(x_0,r).$$
\end{lma}

We are now in a position to prove Theorem~\ref{thm:local-RF}.
We stay as close as possible to the proof of \cite[Theorem 5.1]{LeeTopping2022}.

\begin{proof}[Proof of Theorem~\ref{thm:local-RF}]
We start by specifying the positive constants that will be used in the construction.
\begin{itemize}
\item $\Lambda(n)>1$ and $\a(n)$ from Proposition~\ref{construction-local-RF-new};
\item $\b(n)$ from Lemma~\ref{l-balls};
\item $C_0(n,b,s_0)$ from Proposition~\ref{cur-estimate};
\item $a_0(n,b,s_0)=\max\{1, \a\Lambda,\Lambda(\a+C_0) \}$;
\item $S_0(n,a_0,b,s_0)$ from Lemma~\ref{preserv-cone}; 
\item $S_1(n,b,s_0)$ from Proposition~\ref{cur-estimate};
\item $\Lambda_0(n,b,s_0) :=\max\{ 4 S_1^{-1/2}a_0^{-1/2}, 4 S_0^{-1/2}a_0^{-1/2},8\b, 8C_0^{-1/2} a_0^{-1/2}\}$
\item $\mu(n,b,s_0):=\sqrt{1+\a C_0^{-1}} -1 >0$
\end{itemize}

\medskip

Choose $\rho\in (0,1)$ sufficiently small so that for all $x\in B_{g_0}(p,4)$, we have 
$|R_{g_0}|\leq \rho^{-2}$. 
We have no uniform positive lower bound for $\rho$.
Apply Proposition~\ref{construction-local-RF-new} with $N=B_{g_0}(p,4)$ to find a smooth solution $g(t)$ to the Ricci flow defined on $B_{g_0}(p,3)\times [0,\a \rho^2]$ with 
$|R_{g(t)}|\leq \Lambda \rho^{-2}$ and $g(0)=g_0$ on $B_{g_0}(p,3)$.  In particular, for all $(x,t)\in B_{g_0}(p,3)\times (0,\a \rho^{2}]$, we have
$$|R_{g(t)}|\leq \Lambda \rho^{-2} \leq a_0 t^{-1}.$$

Now define sequences of times $t_k$ and radii $r_k$ inductively as follows: 
\begin{enumerate}
\item[(a)] $t_1=\a \rho^2$, $r_1=3$ where $\rho$ is obtained from above;
\item[(b)] $t_{k+1}= (1+\mu)^2 t_k$ for all $k\geq 1$;
\item[(c)] $r_{k+1}=r_k -\Lambda_0 \sqrt{a_0t_k}$ for all $k\geq 1$.
\end{enumerate}

Consider the following statement:

\medskip
${\mathcal{P}(k)}$:
There exists a smooth Ricci flow solution $g(t)$ defined on $B_{g_0}(p,r_k)\times [0,t_k]$ with $g(0)=g_0$ such that 
$|R_{g(t)}|\leq a_0t^{-1}$.
\medskip

Clearly, $\mathcal{P}(1)$ is true.  We want to show that $\mathcal{P}(k)$ is true for every $k$ while $r_k>0$, and we do so by induction. 

Suppose $\mathcal{P}(k)$ is true for some $k\in \mathbb{N}$, and consider a Ricci flow $g(t)$ that this provides.  We want to show that $\mathcal{P}(k+1)$ is true if $r_{k+1}>0$ by extending $g(t)$ to a longer time interval. 

\medskip

Let $x\in B_{g_0}(p,r_{k+1}+\frac34 \Lambda_0 \sqrt{a_0t_k})$ so that 
$$B_{g_0}\left(x,\frac14 \Lambda_0\sqrt{a_0t_k}\right)\Subset B_{g_0}(p,r_k).$$
Consider the rescaled Ricci flow $\tilde g(t)=\lambda^{-2} g(\lambda^2 t)$ for $t\in [0,\lambda^{-2}t_k]$ where $\lambda=\frac14 \Lambda_0 \sqrt{a_0t_k}$ so that $B_{g(0)}(x,\frac14 \Lambda_0 \sqrt{a_0t_k})=B_{\tilde g(0)}(x,1)$ and $\lambda^{-2}t_k=16\Lambda_0^{-2}a_0^{-1}$. On the rescaled domain, the Ricci flow $\tilde g(t)$ is smooth and satisfies 
\begin{enumerate}
\item[(i)] $R_{\tilde g(0)} \in \tilde C(b,s_0)$.
\item[(ii)] $|R_{\tilde g(t)}|\leq a_0t^{-1}$ on $B_{\tilde g(0)}(x,1)\times (0,16\Lambda_0^{-2} a_0]$.
\end{enumerate}
Applying Lemma~\ref{preserv-cone} to $\tilde g(t)$, we deduce that
$$R_{\tilde g(t)}(x)+I \in  \tilde C(b,s_0) $$
for $t\leq \min\{S_0, 16\Lambda_0^{-2}a_0^{-1}\}= 16\Lambda_0^{-2}a_0^{-1}$ thanks to the choice of $\Lambda_0$. Hence for all $(x,t)\in B_{g_0}(p,r_{k+1}+\frac34 \Lambda_0 \sqrt{a_0t_k})\times [0,t_k]$, we have 
$$R_{ g(t)}+\left(\frac14 \Lambda_0 \sqrt{a_0t_k} \right)^{-2} I \in \tilde C(b,s_0).$$

We now use this to obtain an improved curvature estimate on a slightly smaller ball. For  $x\in \Omega:=B_{g_0}(p,r_{k+1}+\frac14 \Lambda_0 \sqrt{a_0t_k})$, we have 
$$B_{g_0}\left(x, \frac12\Lambda_0 \sqrt{a_0t_k}\right)\Subset B_{g_0}\left(p,r_{k+1}+\frac34 \Lambda_0\sqrt{a_0t_k}\right).$$

By Lemma~\ref{l-balls} and our choice of $\Lambda_0$,  we deduce that 
$$B_{g(t)}\left(x, \frac14\Lambda_0 \sqrt{a_0t_k}\right)\Subset B_{g_0}\left(p,r_{k+1}+\frac34 \Lambda_0\sqrt{a_0t_k}\right).$$

Therefore, the rescaled Ricci flow $\tilde g(t),t\in [0,16\Lambda_0^{-2}a_0^{-1}]$ satisfies
\begin{enumerate}
\item[(I)]$B_{\tilde g(t)}(x,1)\Subset M$; 
\item[(II)] $R_{\tilde g(t)}+I \in \tilde C(b,s_0)$ on $B_{\tilde g(t)}(x,1)$ 
\end{enumerate}
for each $t\in [0,16\Lambda_0^{-2}a_0^{-1}]$,
and hence Proposition~\ref{cur-estimate} applies to show that 
\begin{equation}\label{improve-estimate}
|R_{\tilde g(t)}(x)|\leq C_0t^{-1}
\end{equation}
 for $0<t\leq \min\{S_1,16\Lambda_0^{-2}a_0^{-1}\}=16\Lambda_0^{-2}a_0^{-1}$. Since this estimate is scaling invariant, we have improved the curvature decay of $g(t)$
from $a_0t^{-1}$ to $C_0t^{-1}$
on 
$\Omega\times (0,t_k]$, where we recall that 
$\Omega:=B_{g_0}(p,r_{k+1}+\frac14 \Lambda_0 \sqrt{a_0t_k})$.
 
Now we construct an extension of $g(t)$.  
For $h_0=g(t_k)$,
\eqref{improve-estimate} implies $\sup_{\Omega}|R_{h_0}|\leq \rho_0^{-2}$ where $\rho_0=\sqrt{C_0^{-1}t_k}$. 
Note that by definition of $\Lambda_0$, we have $\rho_0\leq \frac18 \sqrt{a_0 t_k}\Lambda_0$.
Moreover, for $x\in B_{g_0}(p,r_{k+1})$, Lemma~\ref{l-balls} 
(using only the original $a_0t^{-1}$ curvature decay rather than the refined $C_0t^{-1}$ decay)
and the choice of $\Lambda_0$ imply
 \begin{equation}
 \begin{split}
 B_{g(t_k)}(x,\rho_0) 
 &\subset B_{g_0}(x,\rho_0+\b\sqrt{a_0 t_k})\\
 &\subset B_{g_0}\left(x,\frac14 \Lambda_0\sqrt{a_0t_k}\right) \Subset \Omega.
 \end{split}
 \end{equation}
This shows that $B_{g_0}(p,r_{k+1})\subset \Omega_{\rho_0}$ where 
$\Omega_{\rho_0}$,
as in Proposition \ref{construction-local-RF-new},
is computed using $h_0$. By applying Proposition~\ref{construction-local-RF-new}, we find a Ricci flow $g(t)$ on $B_{g_0}(p,r_{k+1})\times 
[t_k,t_{k+1}]$, 
extending $g(t)$ on this smaller ball, with $t_{k+1}=t_k+\a \rho_0^2=(1+\mu)^2t_k$ and 
\begin{equation}
|R_{g(t)}|\leq \Lambda \rho_0^{-2} =\Lambda C_0 t_k^{-1} 
\leq a_0 t_{k+1}^{-1}  \leq a_0 t^{-1}
\end{equation}
thanks to the choice of $a_0$. This shows that $\mathcal{P}(k+1)$ is true if $r_{k+1}>0$. 
\medskip

Since $\lim_{k\to +\infty} r_k =-\infty$ and $r_1=3$, there exists $i\in\mathbb{N}$ such that $r_i\geq 2$ and $r_{i+1}<2$. Since $\mathcal{P}(i)$ holds, we now wish to estimate the corresponding $t_i$ from below. 
\begin{equation}
\begin{split}
2> r_{i+1} &= 3-\Lambda_0 \sqrt{a_0} \cdot \sum_{k=1}^i \sqrt{t_k}\\
&\geq  3-\Lambda_0 \sqrt{a_0t_i} \cdot \sum_{k=0}^\infty (1+\mu)^{-k}\\
&= 3-\sqrt{t_i} \cdot \frac{\Lambda_0 \sqrt{a_0}(1+\mu)}{\mu}
\end{split}
\end{equation}
which implies 
\begin{equation}
t_i > \frac{\mu^2 }{a_0 \Lambda_0^2(1+\mu)^2 }=:T(n,b,s_0).
\end{equation}

This shows that there exists a smooth Ricci flow solution $g(t)$ on $B_{g_0}(p,2)\times [0,T]$ such that $g(0)=g_0$ and $|R_{g(t)}|\leq a_0t^{-1}$.  The conclusion on pinching follows from applying Lemma~\ref{preserv-cone} on $B_{g_0}(x,1)$ where 
$x\in B_{g_0}(p,1)$ provided that we shrink $T$ further if necessary. This completes the proof.
\end{proof}

\medskip

We can now establish the existence of a Ricci flow on $M\times [0,+\infty)$ 
as claimed in Theorem \ref{Thm:existence},
using Theorem~\ref{thm:local-RF}.

\begin{proof}[Proof of Theorem~\ref{Thm:existence}]
By Lemma~\ref{pinching-implies-cone}, 
the pinching hypothesis \eqref{e0_pinching_hyp} implies that 
$R_{g_0}\in \tilde C(b,s_0)$ for some $s_0>0$ and $b\in (0,\Upsilon_n)$
depending only on $n$ and $\e_0$.

Fix $p\in M$. Pick $R_i\to +\infty$ and denote $h_{i,0}=R_i^{-2}g_0$.
Then
$R_{h_{i,0}} \in \tilde C(b,s_0)$ for all $i$. We  apply Theorem~\ref{thm:local-RF} to $h_{i,0}$ to obtain a Ricci flow solution $h_i(t)$ on $B_{h_{i,0}}(p,1)\times [0,T]$ with 
\begin{enumerate}
\item[(i)] $|R_{h_i(t)}|\leq a_0t^{-1}$;
\item [(ii)] $R_{h_i(t)}+I \in \tilde C(b,s_0)$
\end{enumerate}
for some $a_0(n,s_0),T(n,s_0)>0$. Define the rescaled Ricci flow solution $g_i(t)=R_i^2 h_i(R_i^{-2}t)$ on $B_{g_0}(p,R_i)\times [0,TR_i^2]$  with 
\begin{equation}
\left\{
\begin{array}{ll}
g_i(0)=g_0; \\
|R_{g_i(t)}|\leq a_0 t^{-1};\\
R_{g_i(t)}+R_i^{-2}I \in \tilde C(b,s_0)
\end{array}
\right.
\end{equation}
on each $B_{g_0}(p,R_i)\times (0,TR_i^2]$.

By \cite[Corollary 3.2]{Chen2009} (see also \cite{Simon2008}) and a modification of Shi's higher order estimate given in \cite[Theorem 14.16]{ChowBookII}, we deduce that for any $k\in \mathbb{N}$, $S>0$ and $\Omega\Subset M$, there exists $C(k,\Omega,g_0,a_0,S)>0$ so that for sufficiently large $i$ we have
\begin{align}
\sup_{\Omega\times [0,S]}|\nabla^k R_{g_i(t)}|\leq C(k,\Omega,g_0,a_0,S).
\end{align}
By applying the Ascoli-Arzel\`a theorem in coordinate charts, 
a subsequence converges to a smooth solution $g(t)=\lim_{i\rightarrow +\infty}g_i(t)$ of the Ricci flow on $M\times [0,+\infty)$ so that $g(0)=g_0$, $|R_{g(t)}|\leq a_0t^{-1}$ and
\begin{equation}
R_{g(t)}\in  \tilde C(b,s_0) 
\end{equation}
throughout $M\times[0,+\infty)$.
This implies the pinching conclusion (a) of the theorem by Lemma~\ref{lma:linear-algebra},
because $\tilde C(b,s_0)\subset \ell_{b}(\CPICo)$. 
Moreover, since $(M,g_0)$ is complete, $g(t)$  is a complete solution by Lemma \ref{l-balls}.    This completes the proof.
\end{proof}

Now we are ready to prove the main Theorem \ref{main-thm-updated}, which we restate in the following equivalent form.
\begin{thm} \label{pinching_was_cor}
Suppose, for $n\geq 4$, that $(M^n,g_0)$ is a complete non-compact manifold such that
\begin{enumerate}
\item[(i)] $R_{g_0}-\e_0\, \mathrm{scal}(R_{g_0}) \cdot I\in \CPICo$ for some $\e_0\in (0,1)$;
\item[(ii)] $R_{g_0}\in \CPICt$,
\end{enumerate}
then $(M,g_0)$ is flat.
\end{thm}
Before we begin the proof, we observe the following basic local fact that we will need more than once.
\begin{lma}
\label{prod_lem}
A manifold of dimension at least three that splits isometrically  into 
a nontrivial product cannot be PIC1.
\end{lma}
To clarify, being PIC1 means that the curvature tensor lies in the interior of \CPICo, or alternatively that the complex sectional curvatures corresponding to  PIC1 sections 
are strictly positive. 
\begin{proof}
Consider a manifold $M_1\times M_2$, where $\mathrm{dim}(M_1)\geq 1$
and $\mathrm{dim}(M_2)\geq 2$. Pick a unit vector $v$ in some tangent 
space $T_x M_1$ and two orthonormal vectors $e_1,e_2$ in a tangent space $T_y M_2$. Then the complex vectors $v$ and $\frac{1}{\sqrt{2}}(e_1+ie_2)$ span a PIC1 section with complex sectional curvature
$$\textstyle
R(v,\frac1{\sqrt{2}}(e_1+ie_2), v, \frac1{\sqrt{2}}(e_1-ie_2))
=\frac12(R(v,e_1,v,e_1)+R(v,e_2,v,e_2))=0,$$
which is not strictly positive.
\end{proof}

\begin{proof}[Proof of Theorem \ref{pinching_was_cor}]
By working on the universal cover, we may assume $M$ to be simply connected. Suppose on the contrary, we have $\mathrm{scal}(R_{g_0}(p))>0$ for some $p\in M$. By Theorem~\ref{Thm:existence}, we can find a long-time solution $g(t)$ to the Ricci flow on $M\times[0,+\infty)$ such that 

\begin{enumerate}
\item[(a)] $R_{g(t)}-\e_0' \mathrm{scal}(R_{g(t)}) \cdot I\in \CPICo$ for some  $\e_0'\in (0,\frac{1}{n(n-1)})$;
\item[(b)] $|R_{g(t)}|\leq a_0 t^{-1}$ for some $a_0>0$.
\end{enumerate}

Moreover, it follows from \cite[Theorem 3.1]{LeeTam2020} that $g(t)$ also satisfies $R_{g(t)}\in \CPICt$ for all $t>0$.  Furthermore,  the strong maximum principle implies that $\mathrm{scal}_{g(t)}>0$ for all $t>0$. 
\begin{claim}
We have $\mathrm{K}(g(t))> 0$ for all $t>0$.   That is, all real sectional curvatures are strictly positive for positive times.
\end{claim}

\begin{proof}[Proof of Claim]
Fix a time $t_0>0$ at which to consider the sectional curvatures.
Because $R_{g(t_0)}\in \CPICt$,  
Cabezas-Rivas and Wilking 
\cite[Theorem 5.1]{CabezasWilking2015} tell us that $(M,g(t_0))$ splits isometrically as $\Sigma^k\times F^{n-k}$ where $\Sigma$ is the (closed) soul and 
$F$ is diffeomorphic to $\mathbb{R}^{n-k}$ and carries a complete metric $h$ with $R_h\in \CPICt$.  
Because $\Sigma$ is closed, but $M$ is non-compact, we must have 
$k<n$.
By  (a)  and the positivity of the scalar  curvature, 
$(M,g(t_0))$ is (strictly) PIC1, which is incompatible with being a non-trivial product by Lemma \ref{prod_lem}, so 
$k=0$ and hence $M^n$ is diffeomorphic to $\mathbb{R}^n$. 
Similarly, considering the de-Rham decomposition of $M$,  (a)  implies that $(M,g(t_0))$ is irreducible. 
Meanwhile, if $(M,g(t_0))$ is symmetric, then the scalar curvature is a positive constant, which by   (a)  implies a uniform positive lower bound for the Ricci curvature; Bonnet-Myers then forces $M$ to be compact,
which is assumed not to be the case. 
Hence, $(M,g(t_0))$ is of positive scalar curvature, 
diffeomorphic to $\mathbb{R}^n$, non-symmetric and irreducible. In particular, Berger’s holonomy classification theorem implies that $\mathrm{Hol}(M,g(t_0))$ is either $\mathrm{SO}(n)$ or,
possibly if $n$ is even, $\mathrm{U}(n/2)$. This is because all other options would be Ricci flat or Einstein (hence compact).
If $\mathrm{Hol}(M,g(t_0))=\mathrm{SO}(n)$,  it follows from the strong maximum principle argument in \cite[Proposition 9]{BrendleSchoen2008} that $\mathrm{K}(g(t_0))>0$.
Indeed, if any sectional curvature were zero, then every parallel translation of that section would also have zero sectional curvature, and because the holonomy group is $\mathrm{SO}(n)$ the manifold would have to be flat, violating the positivity of the scalar curvature.

 If $\mathrm{Hol}(M,g(t_0))=\mathrm{U}(n/2)$,  we can still deduce $\mathrm{K}(g(t_0))>0$ as pointed out in the proof of \cite[Corollary 7.6]{CabezasWilking2015}. We include the argument for  convenience.  First, we note that $g(t_0)$ is K\"ahler by the holonomy. Suppose $\mathrm{K}(\sigma)=0$ for some $x\in M$ and real plane $\sigma\subset T_xM$.  If $\sigma$ is a complex holomorphic plane, i.e. $\sigma=\mathrm{span}\{ v,Jv\}$ for some $v\in T_xM$, then the strong maximum principle \cite[Proposition 9]{BrendleSchoen2008} implies that the holomorphic sectional curvature of $g(t_0)$ vanishes on $M$ since $\mathrm{Hol}(M,g(t_0))=\mathrm{U}(n/2)$ and hence $g(t_0)$ is flat \cite[Chapter IX, Proposition 7.1]{KNvol2}
which contradicts  the positive scalar curvature.  Suppose now $\sigma=\mathrm{span}\{u,v\}$ for some $u,v\in T_xM$ where $\{u,v\}$ are orthonormal but $Ju\notin \sigma$. We fix an orthonormal frame $\{e_i,Je_i\}_{i=1}^{n/2}$ 
such that defining 
$u_i:=\frac{1}{\sqrt{2}}(e_i-\sqrt{-1}Je_i)$ makes $\{u_i \}_{i=1}^{n/2}$ 
a unitary frame with $u=e_1$ and $v=\cos \theta \cdot Je_1+\sin \theta \cdot e_2$ for some $\theta\in (0,2\pi)$. Then $K(\sigma)=R(u,v,u,v)=0$ is equivalent to 
\begin{equation}
\label{cos_sin_exp}
\cos^2\theta\cdot R(e_1,Je_1,e_1,Je_1)+\sin 2\theta\cdot  R(e_1,e_2,e_1,Je_1)+\sin^2\theta\cdot R(e_1,e_2,e_1,e_2)=0.
\end{equation}
By considering the  linear transformation of $T_xM$ that 
fixes each $u_i$ for $i\neq 2$, but sends $u_2$ to $-u_2$, which is an element in $\mathrm{U}(n/2)$, we deduce by the strong maximum principle \cite[Proposition 9]{BrendleSchoen2008} that 
\eqref{cos_sin_exp} holds also with the sign of the middle term reversed, and adding both gives
$$\cos^2\theta\cdot R(e_1,Je_1,e_1,Je_1)+\sin^2\theta\cdot R(e_1,e_2,e_1,e_2)=0$$
and hence 
$R(e_1,e_2,e_1,e_2)=0$
since $\mathrm{K}\geq 0$. Similarly, since $u_2\mapsto \sqrt{-1}u_2$ is an element of $\mathrm{U}(n/2)$, we deduce that $R(e_1,Je_2,e_1,Je_2)=0$.
Therefore by the Bianchi identity and K\"ahler symmetries we have
\begin{equation*}
\begin{aligned}
R(u_1,\bar u_1,u_2,\bar u_2) &= - R(e_1,Je_1,e_2,Je_2)\\
&=R(e_1,e_2,Je_2,Je_1)+R(e_1,Je_2,Je_1,e_2)\\
&=-R(e_1,e_2,e_1,e_2)-R(e_1,Je_2,e_1,Je_2)\\
&=0.
\end{aligned}
\end{equation*}
Using once again the invariance under $\mathrm{U}(n/2)$ coming from the strong maximum principle, we deduce
that the orthogonal bisectional curvature of $g(t_0)$ vanishes.
This then contradicts the positivity of the scalar curvature 
because the scalar curvature can be written as an average of orthogonal bisectional curvatures. For example, according to 
a formula of Berger \cite[(2.1)]{NiZheng2018}, one can write the scalar curvature in terms of the orthogonal Ricci curvature $ \Ric^\perp$,
which vanishes when the orthogonal bisectional curvature vanishes.
This completes the proof of the claim.
\end{proof}

Since we have $\mathrm{K}(g(t))>0$ and $|R_{g(t)}|\leq a_0 t^{-1}$ for all $t>0$, a result of Gromoll-Meyer 
\cite[Theorem B.65]{CK04}
implies that there exists $c_0(n,a_0)>0$ such that $\mathrm{inj}(g(t))\geq c_0\sqrt{t}$ for all $t>0$. 
We claim that  the asymptotic volume ratio is positive, i.e. for any $x\in M$, 
$$\mathrm{AVR}(g_0):=
\lim_{r\to\infty}\frac{\mathrm{Vol}_{g_0}\left(B_{g_0}(x,r) \right)}{\omega_n r^n}
>0,$$ 
where $\omega_n$ is the volume of the unit ball in $\R^n$.
For $t_0>0$ sufficiently large,  Lemma~\ref{l-balls} implies that $$B_{g(t_0)}(x,c_0\sqrt{t_0})\subset B_{g_0}(x,c_1\sqrt{t_0})$$ 
for some $c_1(n,a_0)>0$. Additionally, $g(t_0)\leq g_0$ 
because $\Ric_{g(t)}\geq 0$, so
\begin{equation}
\begin{split}
\mathrm{Vol}_{g_0}\left(B_{g_0}(x,c_1\sqrt{t_0}) \right) &\geq \mathrm{Vol}_{g_0}\left(B_{g(t_0)}(x,c_0\sqrt{t_0}) \right)\\
&\geq \mathrm{Vol}_{g(t_0)}\left(B_{g(t_0)}(x,c_0\sqrt{t_0}) \right)\\
&\geq c_2(n,a_0) t_0^{n/2}
\end{split}
\end{equation}
where the last inequality follows by G\"unther's theorem 
(cf. \cite[Theorem 3.101(ii)]{GHL})
because $\mathrm{inj}(g(t_0))\geq c_0\sqrt{t_0}$ and 
$|R_{g(t_0)}|\leq a_0 t_0^{-1}$.

Since $t_0$ is arbitrarily large,  we see that $\mathrm{AVR}(g_0)>0$. It is well-known that the asymptotic volume ratio $\mathrm{AVR}(g)$ is preserved under Ricci flow with $\Ric_{g(t)}\geq 0$ and $|R_{g(t)}|\leq a_0 t^{-1}$; for instance see the proof of \cite[Theorem 7]{TakumiYokota-GD}.
Therefore, $\mathrm{AVR}(g(t))=\mathrm{AVR}(g_0)>0$ for all $t>0$.

On the other hand, since $R_{g(t)}\in \CPICt$, $g(t)$ satisfies the Hamilton's differential Harnack inequality for all $t\geq 0$ by \cite{Brendle2009}. Together with the fact that $t\cdot \mathrm{scal}_{g(t)}$ is uniformly bounded for all $t>0$, we can use an argument 
of Schulze-Simon \cite[Theorem 1.2]{SchulzeSimon2013} to deduce that $(M,g_i(t),p)$ where $g_i(t)=i^{-2}g(i^2 t)$ 
converges sub-sequentially in the Cheeger-Gromov sense to $(M_\infty,g_\infty(t),p_\infty)$ which is an expanding gradient soliton with the same asymptotic volume ratio as $g_0$, see also \cite[Proposition 12]{Brendle2009}. By  (a), the Ricci curvature of $g_\infty(t)$ is pinched and hence $g_\infty(t)$ is flat for all $t>0$ using \cite[Corollary 3.1]{Ni2005}. This implies $\mathrm{AVR}(g_\infty(t))=1$ for all $t>0$ which forces $\mathrm{AVR}(g_0)=1$ and hence $R_{g_0}\equiv 0$ by the rigidity of volume comparison. This contradicts  the non-flatness we have assumed at $p$. This completes the proof.
%
%
%
%
%
\end{proof}

It remains to prove the corollary combining our main PIC1 Pinching Theorem \ref{main-thm-updated} with earlier work.
\begin{proof}[{Proof of Corollary \ref{main-thm-cor}}]
Because $(M,g_0)$ is assumed  to be not everywhere flat, by the PIC1 Pinching Theorem \ref{main-thm-updated} it must be compact.
Without loss of generality we may reduce $\e_0$ so that 
$\e_0\in (0,\frac{1}{n(n-1)})$. 
By  Lemma \ref{converse_lemma}, applied with $\cone=\CPICo$,
the pinching hypothesis implies that there exists $b>0$ such that 
$R_{g_0}\in \ell_{b}(\CPICo)$.
But $\ell_{b}(\CPICo)$ is invariant under the Hamilton ODE by \cite[Proposition 3.2]{BohmWilking2008}, 
so if we start the Ricci flow (which is always possible on a closed manifold) then for later times $t>0$ we still have $R_{g(t)}\in \ell_{b}(\CPICo)$, and thus by Lemma 
\ref{lma:linear-algebra} we have 
$$R_{g(t)}-\delta\,\mathrm{scal}(R_{g(t)})\cdot I\in \CPICo$$
for some $\delta>0$.
Because $g_0$ is assumed not to be flat, for $t>0$ we have $\mathrm{scal}(R_{g(t)})>0$ by the maximum principle, and thus $g(t)$ is strictly PIC1.
The result then follows from Brendle's PIC1 version of the sphere theorem \cite[Theorem 3]{brendlePIC1}.
\end{proof}



\appendix

\section{A geometric interpretation of $\check{C} (s)$}
\label{geometric_interpretation_appendix}

In Section \ref{pinching_cones_sect}, for each $s>0$, $\check{C} (s)$ was defined 
to be the cone of all algebraic curvature tensors $R\in \mathcal{C}_B(\mathbb{R}^n)$ satisfying
\begin{equation*}
\begin{split}
& R_{1313}+\lambda^2 R_{1414}+ \mu^2 R_{2323}+\lambda^2\mu^2 R_{2424}-2\lambda\mu R_{1234}\\
&\qquad +\frac1s(1-\lambda^2)(1-\mu^2)\cdot \mathrm{scal}(R)\geq 0
\end{split}
\end{equation*}
for all orthonormal four-frames $\{e_1,e_2,e_3,e_4\}\subset \mathbb{R}^n$ and $\lambda,\mu\in [0,1]$.  
In this appendix we interpret this definition in terms of complex sectional curvatures and use the insight in order to give a quantitative relationship between $\CPICo$ and 
$\check{C} (s)$.

As in Section \ref{intro_sect}, given an algebraic curvature tensor 
$R\in \mathcal{C}_B(\mathbb{R}^n)$, we can extend by complex linearity and consider complex sectional curvatures of two-complex-dimensional subspaces $\Sigma\subset \C^n$.

Every section $\Sigma$ contains an isotropic vector $v\in \Sigma$, i.e. so that 
$(v,v)=0$,
where $(\cdot,\cdot)$ is the complex linear extension of the standard inner product on $\R^n$.
To see this directly, pick any basis $\tilde v, \tilde w$ of $\Sigma$.
Either $\tilde w$ is isotropic, in which case we set $v=\tilde w$, or we can solve
the quadratic polynomial $(\tilde v + z\tilde w,\tilde v + z\tilde w)=0$, and then set  $v=\tilde v + z\tilde w$.
By scaling, we may assume that $|v|=1$.

Because $v$ is isotropic and of unit length, we can pick orthonormal $e_1,e_2\in \R^n$ so that $v=\frac{1}{\sqrt{2}}(e_1+ie_2)$.
Note that $\bar v=\frac{1}{\sqrt{2}}(e_1-ie_2)$ is orthogonal to $v$, e.g. $\langle v, \bar{v}\rangle=(v,v)=0$.

As stated earlier, we refer to $\Sigma$ as a PIC1 section if $\bar v$ is orthogonal to $\Sigma$. More generally $\bar v$ will lie at an angle $\theta\in [0,\pi/2]$ to $\Sigma$ and defining $\alpha=\cos\theta\in [0,1]$
we can define a unit vector $w\in \Sigma$ orthogonal to $v$ by writing
$$w=\alpha \bar{v} +\sqrt{1-\alpha^2} u$$ 
for some unit $u\in \C^n$ that is orthogonal to both $v$ and $\bar v$.
Because $(v,w)=\alpha$, we see that $\alpha=0$ precisely when 
$\Sigma$ is a PIC1 section. 
Indeed, $\alpha$ can be viewed as a measure of how far $\Sigma$ is from being a PIC1 section that depends only on $\Sigma$ (and 
not on our choices of vectors above) with $\alpha=1$ being precisely the 
case that $\Sigma $ is a real section (that is,
the complexification of a two-real-dimensional plane in $\R^n$, or equivalently a section that contains a vector $v$ such that $\{v,\bar v\}$ gives an orthonormal basis for $\Sigma$). We thus view $\alpha$ as a function on the set of sections.

A calculation then reveals that an alternative characterisation of 
$\check C(s)$ is the cone of all curvature tensors 
$R\in \mathcal{C}_B(\mathbb{R}^n)$ whose complex sectional curvature
satisfies
\begin{equation}
\begin{aligned}
\label{Kquant}
K^\C(\Sigma)+\frac{\alpha(\Sigma)^2}{s}\scal(R)\geq 0.
\end{aligned}
\end{equation}
We see very clearly the inclusions $\CPICt\subset  \check C(s)\subset \CPICo$
mentioned in Section \ref{pinching_cones_sect}.

It might be initially a little surprising that a curvature tensor in 
\CPICo, which is assumed only to have non-negative complex sectional curvature for very special sections (the PIC1 sections) does, in fact, enjoy a lower bound for \emph{all} complex sectional curvatures.
\begin{lma}
\label{Kalpha_lower}
Suppose $R\in \CPICo\subset \mathcal{C}_B(\mathbb{R}^n)$. Then for all complex sections $\Sigma\subset\C^n$ we have
$$K^\C(\Sigma)\geq -C(n)\alpha(\Sigma)\scal(R).$$
\end{lma}
Before we prove Lemma \ref{Kalpha_lower}, we record  the following more basic control
that is equivalent to every $R\in \CPICo\setminus \{0\}$
having positive scalar curvature.

\begin{lma}
\label{lower_R_in_PIC1}
If $R\in \CPICo$, then 
$$|R|\leq C(n)\scal(R).$$
\end{lma}
\begin{proof}
Take any orthonormal basis $\{e_i\}$. Then 
$$\scal(R)=\sum_{i\neq j} R(e_i,e_j,e_i,e_j)
=\frac1{2(n-2)}\sum_{i,j,k\text{ distinct}}[R(e_i,e_j,e_i,e_j)+R(e_i,e_k,e_i,e_k)].
$$
Each term in square brackets can be written
$$R(e_i,e_j,e_i,e_j)+R(e_i,e_k,e_i,e_k)=R(e_i,e_j+ie_k,e_i,e_j-ie_k)=2K^\C(\Sigma)\geq 0,$$
where $\Sigma$ is the PIC1 section spanned by $e_i$ and $\frac{1}{\sqrt{2}}(e_j+ie_k)$.
Because all these terms are non-negative, and sum to $2(n-2)\scal(R)$, we have
$$0\leq R(e_i,e_j,e_i,e_j)+R(e_i,e_k,e_i,e_k)\leq 2(n-2)\scal(R).$$
To control the sectional curvature of a plane spanned by $e_i,e_j$ for arbitrary $i\neq j$, we pick any $k\neq i,j$ and compute
\begin{equation}
\begin{aligned}
2 R(e_i,e_j,e_i,e_j) &=
[R(e_i,e_j,e_i,e_j)+R(e_i,e_k,e_i,e_k)]\\
&\quad +[R(e_j,e_i,e_j,e_i)+R(e_j,e_k,e_j,e_k)]\\
&\quad -[R(e_k,e_j,e_k,e_j)+R(e_k,e_i,e_k,e_i)].
\end{aligned}
\end{equation}
Therefore
$$-(n-2)\scal(R)\leq R(e_i,e_j,e_i,e_j) \leq 2(n-2)\scal(R).$$
\end{proof}

\begin{proof}[Proof of Lemma \ref{Kalpha_lower}]
We are interested in the complex sectional curvature corresponding to $\Sigma$, which is spanned by the orthonormal basis $\{v,w\}$ of the type considered earlier.
We compute
\begin{equation*}
\begin{aligned}
K^\C(\Sigma)&=
R(v,w,\bar{v},\bar{w})\\
&=R\left(v,\alpha \bar{v} +\sqrt{1-\alpha^2}\, u,\bar{v},\alpha v +\sqrt{1-\alpha^2}\, \bar{u} \right)\\
&= \alpha^2 R(v,\bar{v},\bar{v},v)+
(1-\alpha^2) R(v,u,\bar{v},\bar{u})\\
&\quad +\alpha\sqrt{1-\alpha^2} R(v,\bar{v}, \bar{v},\bar{u}) 
+ {\alpha}\sqrt{1-\alpha^2} R(v,u,\bar{v},v)\\
&\geq (1-\alpha^2) R(v,u,\bar{v},\bar{u}) - C(n)\alpha\,\scal(R),
\end{aligned}
\end{equation*}
where we have used Lemma \ref{lower_R_in_PIC1}.

Now note that the section spanned by $v$ and $u$ is a PIC1 section because $\bar v$ is orthogonal to both $v$ (because $v$ is isotropic) and $u$ (by construction).
Therefore $R(v,u,\bar{v},\bar{u})\geq 0$ because $R\in \CPICo$.

We conclude that 
$$R(v,w,\bar{v},\bar{w})\geq - C(n)\alpha\, \scal(R),$$
as required.
\end{proof}

\begin{cor}
\label{PIC1_to_checkC}
For every $n\geq 4$ and $\e>0$ there exists $s_0>0$ such that for every $R\in\CPICo$ we have
$$R+\e\,\scal(R)\cdot I \in \check C(s_0).$$
\end{cor}

\begin{proof}
Take an arbitrary section $\Sigma\subset \C^n$. Then by Lemma \ref{Kalpha_lower}
we have
\begin{equation*}
\begin{aligned}
K^\C(\Sigma) &\geq -C(n)\alpha(\Sigma)\scal(R)\\
&\geq -\e\, \scal(R)-c_1(n,\e)\alpha(\Sigma)^2\scal(R),\\
\end{aligned}
\end{equation*}
for some $c_1(n,\e)>0$, by Young's inequality.
If we write $\hat R:=R+\e\,\scal(R)\cdot I $, then 
$\scal(\hat R)\geq \scal(R)$, so we find that
$$K_{\hat R}^\C(\Sigma)+c_1(n,\e)\alpha(\Sigma)^2\scal(\hat R)\geq 0,$$
and we can choose $s_0=\frac{1}{c_1(n,\e)}$ to deduce that $\hat R\in \check C(s_0)$.
\end{proof}


\end{document}